\documentclass[a4paper, 11pt]{article}

\usepackage{amsmath}
\usepackage{amsfonts}
\usepackage{amssymb}
\usepackage[english]{babel}
\usepackage{graphicx}
\usepackage{amsthm}

\newtheorem{theorem}{Theorem}
\newtheorem{proposition}[theorem]{Proposition}
\newtheorem{lemma}[theorem]{Lemma}
\newtheorem{definition}{Definition}
\newtheorem{corollary}[theorem]{Corollary}

\begin{document}
\title{A novel characterization of cubic Hamiltonian graphs via the
associated quartic graphs}

\author{S. Bonvicini\thanks{Department of Sciences and Methods for Engineering,
University of Modena and Reggio Emilia, Italy,
simona.bonvicini@unimore.it}, T. Pisanski\thanks{University of
Primorska, Koper and University of Ljubljana, Ljubljana, Slovenia,
Tomaz.Pisanski@fmf.uni-lj.si}}

\maketitle

\begin{abstract}
We give a necessary and sufficient condition for a cubic graph to be
Hamiltonian by analyzing Eulerian tours in certain spanning
subgraphs of the quartic graph associated with the cubic graph by
$1$-factor contraction. This correspondence is most useful in the
case when it induces a blue and red $2$-factorization of the
associated quartic graph. We use this condition to characterize the
Hamiltonian $I$-graphs, a further generalization of generalized
Petersen graphs. The characterization of Hamiltonian $I$-graphs
follows from the fact that one can choose a $1$-factor in any
$I$-graph in such a way that the corresponding associated quartic
graph is a graph bundle having a cycle graph as base graph and a
fiber and the fundamental factorization of graph bundles playing the
role of blue and red factorization. The techniques that we develop
allow us to represent Cayley multigraphs of degree $4$, that are
associated to abelian groups, as graph bundles. Moreover, we can
find a family of connected cubic (multi)graphs that contains the
family of connected $I$-graphs as a subfamily.

\vspace{0.2cm}\noindent {\bf Keywords:} generalized Petersen graphs,
$I$-graphs, Hamiltonian cycles,  Eulerian tours, Cayley multigraphs.
\vspace{0.2cm}\noindent {\bf MSC 2000:} 05C45, 05C25, 05C15, 05C76,
05C70, 55R10, 05C60.

\end{abstract}

\section{Introduction.}\label{intro}

A graph is \emph{Hamiltonian} if it contains a spanning cycle
(\emph{Hamiltonian cycle}). To find a Hamiltonian cycle in a graph
is an NP--complete problem (see \cite{GJT}). This fact implies that
a characterization result for Hamiltonian graphs is hard to find.
For this reason, most graph theorists have restricted their
attention to particular classes of graphs.

In this paper we consider cubic graphs. In Section
\ref{sec_transitions} we give a necessary and sufficient condition
for a cubic graph to be Hamiltonian. Using this condition we can
completely characterize the Hamiltonian $I$-graphs.

The family of $I$-graphs is a generalization of the family of
generalized Petersen graphs. In \cite{BCMS}, the generalized
Petersen graphs were further generalized to $I$-graphs. Let $n$,
$p$, $q$ be positive integers, with $n\geq 3$, $1\leq p, q\leq n-1$
and $p, q\neq n/2$. An $I$-graph $I(n, p, q)$ has vertex-set $V(I(n,
p, q))=\{v_i, u_i: 0\leq i\leq n-1\}$ and edge-set $E(I(n, p,
q))=\{[v_i, v_{i+p}], [v_i, u_i], [u_i, u_{i+q}] : 0\leq i\leq
n-1\}$ (subscripts are read modulo $n$). The graph $I(n, p, q)$ is
isomorphic to the graphs $I(n, q, p)$, $I(n, n-p, q)$ and $I(n, p,
n-q)$. It is connected if and only if $\gcd(n, p, q)=1$ (see
\cite{BPZ}).

For $p=1$ the $I$-graph $I(n, 1, q)$ is known as a generalized
Petersen graph and is denoted by $G(n, q)$. The Petersen graph is
$G(5, 2)$. It has been proved that $I(n, p, q)$ is isomorphic to a
generalized Petersen graph if and only if $\gcd(n, p)=1$ or $\gcd(n,
q)=1$ (see \cite{BPZ}). A connected $I$-graph which is not a
generalized Petersen graph is called a \emph{proper} $I$-graph.
Recently, the class of $I$-graphs has been generalized to the class
of GI-graphs (see \cite{MPZ}).

It is well known that the Petersen graph is not Hamiltonian. A
characterization of Hamiltonian generalized Petersen graphs was
obtained by Alspach \cite{Al}.

\begin{theorem}[Alspach, \cite{Al}]\label{th_Alspach}
A generalized Petersen graph $G(n, q)$ is Hamiltonian if and only if
it is not isomorphic to $G(n, 2)$ when $n\equiv 5\pmod 6$.
\end{theorem}

In this paper we develop a powerful theory that helps us extend this
result to all $I$-graphs.

\begin{theorem}\label{th_hamiltonian}
A connected $I$-graph is Hamiltonian if and only if it is not
isomorphic to $G(n, 2)$ when $n\equiv 5\pmod 6$.
\end{theorem}

For the proof the above main theorem, we developed techniques that
are of interest by themselves and are presented in the following
sections. In particular, we introduce \emph{good Eulerian graphs}
that are similar to \emph{lattice diagrams} that were originally
used by Alspach in his proof of Theorem \ref{th_Alspach}.

Our theory also involves Cayley multigraphs. In Section
\ref{sec_graphs_Xstr} we show that Cayley multigraphs of degree $4$,
that are associated to abelian groups, can be represented as graph
bundles \cite{PV}. By the results concerning the isomorphisms
between Cayley multigraphs (see \cite{DeFaMa}), we can establish
when two graph bundles are isomorphic or not (see Section
\ref{sec_isoX}). Combining the definition of graph bundles with
Theorem \ref{th_cubic_quartic}, we can find a family of connected
cubic (multi)graphs that contains the family of connected $I$-graphs
as a subfamily (see Section \ref{sec_generalizing_Igraphs}).

\section{Cubic graph with a $1$-factor and the associated quartic graph with
transitions.}\label{sec_transitions}

A cubic Hamiltonian graph has a $1$-factor. In fact, it has at least
three (edge-disjoint) $1$-factors. Namely any Hamiltonian cycle is
even and thus gives rise to two $1$-factors and the remaining chords
constitute the third $1$-factor. The converse is not true. There are
cubic graphs, like the Petersen graph, that have a $1$-factor but
are not Hamiltonian. Nevertheless, we may restrict our search for
Hamiltonian graphs among the cubic graphs to the ones that possess a
$1$-factor. In this section, we give a necessary an sufficient
condition for the existence of a Hamiltonian cycle in a cubic graph
$G$ possessing a $1$-factor $F$.

\smallskip

Let $G$ be a connected cubic graph and let $F$ be one of its
$1$-factors. Denote by $X=G/F$ the graph obtained from $G$ by
contracting the edges of $F$. The graph $X$ is connected quartic,
i.e. regular of degree $4$ and might have multiple edges. We say
that the quartic graph $X$ is \emph{associated with $G$ and $F$}.
Since $X$ is even and connected, it is Eulerian. A path on three
vertices with middle vertex $v$ that is a subgraph of $X$ is called
a \emph{transition at} $v$. Since any pair of edges incident at $v$
defines a transition, there are $\binom 42=6$ transitions at each
vertex of $X$. For general graphs each vertex of valence $d$ gives
rise to $\binom d2$ transitions. In an Euler tour some transitions
may be used, others are not used. We are interested in some
particular Eulerian spanning subgraphs $W$. Note that any such graph
is sub-quartic and the valence at any vertex of $W$ is either $4$ or
$2$. A vertex of valence $4$ has therefore $6$ transitions, while
each vertex of valence $2$ has $\binom 22=1$ transition. Let $Y$ be
the complementary $2$-factor of $F$ in $G$. Note that the edges of
$Y$ are in one-to-one correspondence with the edges of $X$, while
the edges of $F$ are in one-to-one correspondence with the vertices
of $X$. If $a$ is an edge of $Y$, we denote by $a'$ the
corresponding edge in $X$. If $e$ is an edge of $F$ the
corresponding vertex of $X$ will be denoted by $x_e$. Let $u$ and
$v$ be the end-vertices of $e$ and let $a$ and $b$ be the other
edges incident with $u$ and similarly $c$ and $d$ the edges incident
with $v$. After contraction of $e$, the vertex $x_e$ is incident
with four edges: $a'$, $b'$, $c'$, $d'$. By considering the pre
images of the six transitions at $x_e$, they fall into two disjoint
classes. Transitions $a'b'$ and $c'd'$ are \emph{non-traversing}
while the other four transitions are \emph{traversing transitions}.
In the latter case the edge $e$ has to be used to traverse from one
edge of the pre image transition to the other.

Let $W$ be a spanning Eulerian sub-quartic subgraph of $X$.
Transitions of $X$ carry over $W$. $4$-valent vertices of $W$ keep
the same six transitions, while each $2$-valent vertex inherits a
single transition. We say that $W$ is \emph{admissible} if the
transition at each $2$-valent vertex of $W$ is traversing.

Let $W$ be an admissible subgraph of $X$. A tour in $W$ that allows
only non-traversing transitions at each $4$-valent vertex of $W$ is
said to be a \emph{tour with allowed transitions}. Note that a tour
with allowed transitions uses traversing transitions at each
$2$-valent vertex of $W$. (We recall that a tour in a graph is a
closed walk that traverses each edge of $G$ at least once
\cite{BonMur}). A tour with allowed transitions might have more than
one component.

\begin{lemma}\label{lemma_2factor}
Let $G$ be a connected cubic graph with $1$-factor $F$. There is a
one-to-one correspondence between $2$-factors $T$ of $G$ and
admissible Eulerian subgraphs $W$ of $X=G/F$ in such a way that the
number of cycles of $T$ is the same as the number of components of a
tour with allowed transitions in $W$.
\end{lemma}

\begin{proof} Let $T$ be a $2$-factor of $G$ and let $e=uv$ be an edge
of the $1$-factor $F$. Let $W$ be the projection of $T$ to $X=G/F$.
We will use the notation introduced above. Hence the edge $e$ and
its end-vertices $u$ and $v$ project to the same vertex $x_e$ of
$X$. There are two cases:

Case $1$: $e$ belongs to $T$. In this case exactly one other edge,
say $a$, incident with $u$ and another edge, say $c$, incident with
$v$ belong to $T$. The other two edges ($b$ and $d$) do not belong
to $T$. This means that $x_e$ is a $2$-valent vertex with traversing
transition.

Case $2$: $e$ does not belong to $T$. In this case both edges $a$
and $b$ incident with $u$ belong to $T$ and both edges $c$ and $d$
incident with $v$ belong to $T$. In this case $x_e$ is a $4$-valent
vertex with non-traversing transitions.

Clearly, $W$ is an admissible Eulerian subgraph. Each component of
the tour determined by $W$ with transitions gives back a cycle of
$T$. The correspondence between $T$ and $W$ is therefore
established.\end{proof}


An Eulerian tour in $W$ with allowed transitions is said to be
\emph{good}. An admissible subgraph $W$ of $X$ possessing a good
Eulerian tour is said to be a \emph{good Eulerian subgraph}. In a
good Eulerian subgraph $W$ there are two extreme cases:
\begin{enumerate}
\item each vertex of $W$ is $4$-valent: this means that $W=X$; in
this case the complementary $2$-factor $Y=G-F$ is a Hamiltonian
cycle and no edge of $F$ is used;

\item each vertex of $W$ is $2$-valent: this means that $W$ is a
good Hamiltonian cycle in $X$. In this case $F$ together with the
pre images of edges of $W$ in $G$ form a Hamiltonian cycle.
\end{enumerate}

\begin{theorem}\label{th_eulerian_subgraphs}
Let $G$ be a connected cubic graph with $1$-factor $F$. Then $G$ is
Hamiltonian if and only if $X=G/F$ contains a good Eulerian subgraph
$W$.
\end{theorem}

\begin{proof} Clearly $G$ is Hamiltonian if and only if it contains
a $2$-factor with a single cycle. By Lemma \ref{lemma_2factor}, this
is true if and only if $W$ is an admissible Eulerian subgraph
possessing an Eulerian tour with allowed transitions. But this means
$W$ is good.\end{proof}

\begin{corollary}\label{cor_NP}
Let $G$ be a connected cubic graph with $1$-factor $F$. Finding a
good Eulerian subgraph $W$ of $X=G/F$ is NP-complete.
\end{corollary}

\begin{proof} Since finding a good Eulerian subgraph is equivalent
to finding a Hamiltonian cycle in a cubic graph, and the latter is
NP-complete \cite{GJT}, the result follows readily.\end{proof}


Also in \cite{Fl} Eulerian graphs are used to find a Hamiltonian
cycle (and other graph properties), but our method is different.

The results of this section may be applied to connected $I$-graphs.
The obvious 1-factor $F$ of an $I$-graph $I(n,p,q)$ consists of
spokes. Let $Q(n,p,q)$ denote the quotient $I(n,p,q)/F$. We will
call $Q(n,p,q)$ the \emph{quartic graph associated} with $I(n,p,q)$.

\begin{corollary}\label{cor_eulerian_subgraphs}
Let $I(n,p,q)$ be a connected $I$-graph and let $Q(n,p,q)$  be its
associated quartic graph. Then $I(n,p,q)$ is Hamiltonian if and only
if $Q(n,p,q)$ contains a good Eulerian subgraph $W$.
\end{corollary}

\section{Special $1$-factors and their applications.}\label{sec_special}

Let $G$ be a cubic graph, $F$ a $1$-factor and $Y$ the complementary
$2$-factor of $F$ in $G$. Define an auxiliary graph $Y(G, F)$ having
cycles of $Y$ as vertices and having two vertices adjacent if and
only if the corresponding cycles of $Y$ are joined by one ore more
edges of $F$. If an edge of $F$ is a chord in one of the cycles of
$Y$, then the graph $Y(G, F)$ has a loop. We shall say that the
$1$-factor $F$ is \emph{special} if the graph $Y(G, F)$ is
bipartite. A cubic graph with a special $1$-factor will be called
\emph{special}. If $F$ is a special $1$-factor of $G$, then the
edges of $F$ join vertices belonging to distinct cycles of $Y$,
since $Y(G, F)$ is loopless.

\begin{theorem}\label{th_red_blue_factorization}
Let $G$ be a connected cubic graph and let $F$ be one of its
$1$-factors and $X=G/F$ its associated quartic graph. Then $X$
admits a $2$-factorization into a blue and red $2$-factor in such a
way that the traversing transitions are exactly color-switching and
non-traversing transitions are color-preserving if and only if $G$
and $F$ are special.
\end{theorem}

\begin{proof} Assume that $F$ is a special $1$-factor of $G$. Since
$Y(G, F)$ is bipartite, we can bicolor the vertices of the
bipartition: let one set of the bipartition be blue and the other
red. This coloring induces a coloring on the edges of $Y$: for every
blue vertex (respectively, red vertex) of $Y(G, F)$ we color in blue
(respectively, in red) the edges of the corresponding cycle of $Y$.
Since the edges of $Y$ are in one-to-one correspondence with the
edges of $X$, we obtain a $2$-factorization of $X$ into a blue and
red $2$-factor. Since $F$ is special, the edges of $F$ are incident
with vertices of $G$ belonging to cycles of $Y$ with different
colors (a blue cycle and a red cycle). Therefore, a traversing
transitions is color-switching and a non-traversing transition is
color-preserving.

Conversely, assume that $X$ has a blue and red $2$-factorization
such that the traversing transitions are color-switching and
non-traversing transitions are color-preserving. Since the edges of
$X$ are in one-to-one correspondence with the edges of $Y$, we can
partition the cycles of $Y$ into red cycles and blue cycles. Since
the traversing transitions are color-switching and non-traversing
transitions are color-preserving, the edges of $F$ are incident with
edges belonging to cycles of different colors. This means that the
graph $Y(G, F)$ is bipartite, hence $F$ is special.\end{proof}

\begin{proposition}\label{pro_admissible}
Let $G$ and $F$ be special and let $W$ be any Eulerian subgraph of
$X=G/F$ the associated quartic graph with a blue and red
$2$-factorization. Then $W$ is admissible if and only if each
$2$-valent vertex is incident with edges of different colors.
\end{proposition}

\begin{proof} An Eulerian subgraph $W$ is admissible if and only if
each $2$-valent vertex $v$ in $W$ is incident with edges forming a
traversing transition at $v$. By Theorem
\ref{th_red_blue_factorization}, a traversing transition is
color-switching. Hence, $W$ is admissible if and only if the edges
incident with $v$ have different colors.\end{proof}


Note that quartic graphs with a given $2$-factorization can be put
into one-to-one correspondence with special cubic graphs.

\begin{theorem}\label{th_cubic_quartic}
Any special cubic graph $G$ with a special $1$-factor $F$ gives rise
to the associated quartic graph with a blue and red
$2$-factorization. However, any quartic graph with a given
$2$-factorization determines back a unique special cubic graph by
color-preserving splitting vertices and inserting a special
$1$-factor.
\end{theorem}

\begin{proof} By Theorem \ref{th_red_blue_factorization}, a special
cubic graph $G$ with a special $1$-factor $F$ gives rise to the
graph $X=G/F$ admitting a blue and red $2$-factorization.

Conversely, it is well known that every quartic graph $X$ possesses
a $2$-factorization, i.e. the edges of $X$ can be partitioned into a
blue and red $2$-factor. We use the blue and red $2$-factors of $X$
to construct a cubic graph $G$ as follows: put in $G$ a copy of the
blue $2$-factor and a copy of the red $2$-factor; construct a
$1$-factor of $G$ by joining vertices belonging to distinct copies.
It is straightforward to see that $G$ and $F$ are
special.\end{proof}


We will now apply this theory to the $I$-graphs. In Section
\ref{sec_characterization_ham_Igraph} we will see that this theory
allow us to find a Hamiltonian cycle in a proper $I$-graph and also
to find a family of special cubic graphs that contains the family of
$I$-graphs as a subfamily (see Section
\ref{sec_generalizing_Igraphs}).

Let $I(n, p, q)$ be an $I$-graph. A vertex $v_i$ (respectively,
$u_i$) is called an \emph{outer vertex} (respectively, an
\emph{inner vertex}). An edge of type $[v_i, v_{i+p}]$
(respectively, of type $[u_i, u_{i+q}]$) is called an \emph{outer
edge} (respectively, an \emph{inner edge}). An edge $[v_i, u_i]$ is
called a \emph{spoke}. The spokes of $I(n, p, q)$ determine a
$1$-factor of $I(n, p, q)$. The set of outer edges is said the
\emph{outer rim}, the set of inner edges is said the \emph{inner
rim}. As a consequence of the results proved in \cite{BPZ}, the
following holds.

\begin{proposition}\label{pro8}
Let $I(n, p, q)$, $n\geq 3$, $1\leq p, q\leq n-1$, $p, q\neq n/2$,
be an $I$-graph. Set $t=\gcd(n, q)$ and $s=n/t$. Then $t< n/2$ and
$3\leq s\leq n$. Moreover $I(n, p, q)$ is connected if an only if
$\gcd(t, p)=1$ and $\gcd(s, p)$ is coprime with $q$. It is proper if
and only if $t$ and $\gcd(s, p)$ are different from $1$.
\end{proposition}

\begin{proof} The integer $t$ satisfies the inequality $t<n/2$, since $t$ is a
divisor of $q$ and $q\leq n-1$, $q\neq n/2$; whence $3\leq s\leq n$.
By the results proved in \cite{BPZ}, $I(n, p, q)$ is connected if
and only if $\gcd(n, p, q)=1$. Since $n=st$ and $q=t(q/t)$, the
relation $\gcd(n, p, q)=1$ can be written as $\gcd(st, p,
t(q/t))=1$, whence $\gcd(t, p)=1$ and $\gcd(s, p)$ is coprime with
$q$. Also the converse is true, therefore $I(n, p, q)$ is connected
if and only if $\gcd(t, p)=1$ and $\gcd(s, p)$ is coprime with $q$.
A connected $I$-graph $I(n, p, q)$ is a generalized Petersen graph
if and only if $\gcd(n, q)=1$ or $\gcd(n, p)=1$ (see \cite{BPZ}). By
the previous results, $I(n, p, q)$ is a generalized Petersen graph
if and only if $1=\gcd(n, q)=t$ or $1=\gcd(n,
p)$$=\gcd(st,p)$$=\gcd(s, p)$. The assertion follows.\end{proof}


The smallest proper $I$-graphs are $I(12, 2, 3)$ and $I(12, 4, 3)$.
It is straightforward to see that the following result holds.

\begin{lemma}\label{pro_Igraph_special}
Let $F$ be the $1$-factor determined by the spokes of $I(n, p, q)$
and $X = Q(n, p, q)$ its associated quartic graph. Then $F$ is
special, the graph $X$ is a circulant multigraph $Cir(n; p, q)$, the
blue edges of $X$ correspond to the inner rim and the red edges to
the outer rim of $I(n, p, q)$.
\end{lemma}


In the next section we introduce a class of graphs $X(s,t,r)$ and
later show that it contains $Q(n,p,q)$ as its subclass.

\section{Graphs $X(s, t, r)$.}\label{sec_graphs_Xstr}

Let $\Gamma$ be a group in additive notation with identity element
$0_{\Gamma}$. Let $S$ be a list of not necessarily distinct elements
of $\Gamma$ satisfying the symmetry property $S=-S=\{-\gamma :
\gamma\in S\}$. The Cayley multigraph associated with $\Gamma$ and
$S$, denoted by $Cay(\Gamma, S)$, is an undirected multigraph having
the elements of $\Gamma$ as vertices and edges of the form $[x,
x+\gamma]$ with $x\in\Gamma$, $\gamma\in S$. If $\Gamma$ is a cyclic
group of order $n$, then $Cay(\Gamma, S)$ is a circulant multigraph
of order $n$. A Cayley multigraph $Cay(\Gamma, S)$ is regular of
degree $|S|$ (in determining $|S|$, each element of $S$ is
considered according to its multiplicity in $S$). It is connected if
an only if $S$ is a set of generators of $\Gamma$. If the elements
of $S$ are pairwise distinct, then $Cay(\Gamma, S)$ is a simple
graph and we will speak of Cayley graph. We are interested in
connected Cayley multigraphs of degree $4$. In this case we write
$S$ as the list $S=\{\pm\gamma_1, \pm\gamma_2\}$. A circulant
multigraph of order $n$ will be denoted by $Cir(n; \pm\gamma_1,
\pm\gamma_2)$. If $\gamma_i$, with $i\in\{1, 2\}$, is an involution
of $\Gamma$ or the trivial element, then $\pm\gamma_i$ means that
the element $\gamma_i$ appears twice in the list $S$. Consequently,
the associated Cayley multigraph has multiple edges or loops. We
will denote by $o(\gamma_i)$ the order of $\gamma_i$. We will show
that the Cayley multigraphs $Cay (\Gamma, \{\pm\gamma_1,
\pm\gamma_2\})$ defined on a suitable abelian group $\Gamma$ (and in
particular the circulant multigraphs $Cir(n; \pm\gamma_1,
\pm\gamma_2)$) can be given a different interpretation in terms of
$X(s, t, r)$ graphs (see Figure \ref{fig_graph_Xstr}) defined as
follows.

\begin{definition}\label{defX}
Let $s, t\geq 1$, $0\leq r\leq s-1$ be integers. Let $X(s, t, r)$ be
the graph with vertex-set $\{x^i_j : 0\leq i\leq t-1, 0\leq j\leq
s-1\}$ and edge-set $\{[x^i_j, x^i_{j+1}]: 0\leq i\leq t-1, 0\leq
j\leq s-1\}\cup$ $\{[x^i_j, x^{i+1}_j] : 0\leq i\leq t-2, 0\leq
j\leq s-1\}\cup$ $\{[x^{t-1}_j, x^{0}_{j+r}] : 0\leq j\leq s-1\}$
(the superscripts are read modulo $t$, the subscripts are read
modulo $s$).
\end{definition}

The graph $X(s, t, r)$ has edges of type $[x^i_j, x^i_{j+1}]$,
$[x^i_j, x^{i+1}_j]$ or $[x^{t-1}_j, x^{0}_{j+r}]$.  An edge of type
$[x^i_j, x^i_{j+1}]$ will be called \emph{horizontal}. An edge of
type $[x^i_j, x^{i+1}_j]$ will be called \emph{vertical}, an edge of
$[x^{t-1}_j, x^{0}_{j+r}]$ will be called \emph{diagonal}. For
$t=1$, we say that the edges are horizontal and diagonal (a diagonal
edge is an edge of type $[x^{0}_j, x^{0}_{j+r}]$). For $s=1$ or $(t,
r)=(1, 0)$, each diagonal edge is a loop. For $s=2$ or $s>2$ and
$(t, r)$=$(1, 1)$, $(1, s/2)$, $(2,0)$ the graph has multiple edges.
For the other values of $s$, $t$, $r$, the graph $X(s, t, r)$ is a
simple graph. A simple graph $X(s, t, r)$ is a graph bundle with a
cycle fiber $C_s$ over a cycle base $C_t$; the parameter $r$
represents an automorphism of the cycle $C_s$ that shifts the cycle
$r$ steps (see \cite{PV} for more details on graph bundles). In
literature a simple graph $X(s, t, r)$ is also called
$r$-pseudo-cartesian product of two cycles (see for instance
\cite{FaLiLi}). The definition of $X(s, t, r)$ suggests that the
graph $X(s, t, r)$ is isomorphic to $X(s, t, s-r)$. The existence of
this isomorphism can be also obtained from the following statement.

\begin{figure}[htbp]
\centering
\begin{minipage}[c]{11cm}
\includegraphics[width=11cm]{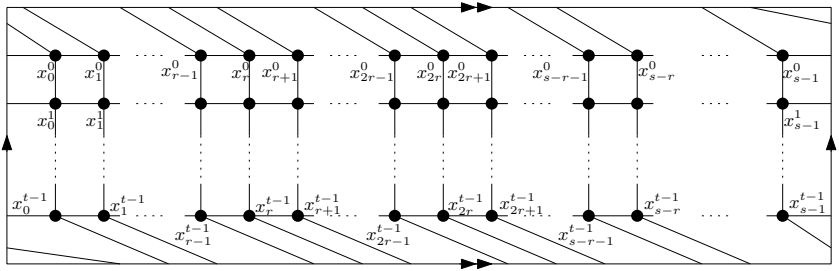}
\end{minipage}
\caption{The graph $X(s, t, r)$ is embedded into torus with
quadrilateral faces; it has a blue and red $2$-factorization: the
vertical and diagonal edges form the blue $2$-factor, the horizontal
edges form the red $2$-factor.}\label{fig_graph_Xstr}
\end{figure}

\begin{proposition}\label{pro_cayley}
Let $Cay(\Gamma, \{\pm\gamma_1, \pm\gamma_2\})$ be a connected
Cayley multigraph of degree $4$, where $\Gamma$ is an abelian group,
$o(\gamma_1)=s$ and $|\Gamma|/s=t$. Then $a\gamma_2=r\gamma_1$ for
some integer $r$, $0\leq r\leq s-1$, if and only if $a=t$.
Consequently, $Cay(\Gamma, \{\pm\gamma_1, \pm\gamma_2\})$ can be
represented as the graph $X(s, t, r)$ or $X(s, t, s-r)$.
\end{proposition}

\begin{proof} We show that $G_1=Cay(\Gamma, \{\pm\gamma_1, \pm\gamma_2\})$ and
$G_2=X(s, t, r)$ are isomorphic. Since $\gamma_1$ and $\gamma_2$ are
generators of $\Gamma$, the elements of $\Gamma$ can be written in
the form $i\gamma_2+j\gamma_1$, where
$i\gamma_2\in\langle\gamma_2\rangle$,
$j\gamma_1\in\langle\gamma_1\rangle$. Hence we can describe the
elements of $\Gamma$ by the left cosets of the subgroup
$\langle\gamma_1\rangle$ in $\Gamma$. By this representation, we can
see that the endvertices of an edge $[x, x\pm\gamma_1]$ of
$Cay(\Gamma, \{\pm\gamma_1, \pm\gamma_2\})$ belong to the same left
coset of $\langle\gamma_1\rangle$ in $\Gamma$; the endvertices of an
edge $[x, x\pm\gamma_2]$ belong to distinct left cosets of
$\langle\gamma_1\rangle$ in $\Gamma$. Therefore,
$a\gamma_2=r\gamma_1\in\langle\gamma_1\rangle$ if and only if $a=t$,
since $Cay(\Gamma, \{\pm\gamma_1, \pm\gamma_2\})$ is connected and
$|\Gamma/\langle\gamma_1\rangle|=t$. Hence we can set
$V(G_1)=\{i\gamma_2+j\gamma_1: 0\leq i\leq t-1, 0\leq j\leq s-1\}$
and $E(G_1)=\{[i\gamma_2+j\gamma_1, (i+1)\gamma_2+j\gamma_1],
[i\gamma_2+j\gamma_1, i\gamma_2+(j+1)\gamma_1]: 0\leq i\leq t-1,
0\leq j\leq s-1\}$. The map $\varphi: V(G_1)\to V(G_2)$ defined by
$\varphi (i\gamma_2+j\gamma_1)=x^i_j$ is a bijection between
$V(G_1)$ and $V(G_2)$. Moreover, if $v_1$, $v_2$ are adjacent
vertices of $G_1$, i.e., $v_1=i\gamma_2+j\gamma_1$ and
$v_2=(i+1)\gamma_2+j\gamma_1$ (or $v_2=i\gamma_2+(j+1)\gamma_1$),
then $\varphi(v_1)=x^i_j$, $\varphi(v_2)=x^{i+1}_j$ (or $\varphi
(v_2)=x^i_{j+1}$) are adjacent vertices of $G_2$. In particular, if
$v_1=(t-1)\gamma_2+j\gamma_1$ and
$v_2=t\gamma_2+j\gamma_1$$=r\gamma_1+j\gamma_1$$=(r+j)\gamma_1$,
then $\varphi(v_1)=x^{t-1}_j$, $\varphi(v_2)=x^{0}_{r+j}$ are
adjacent vertices of $G_2$. It is thus proved that $\varphi$ is an
isomorphism between $G_1$ and $G_2$. If we replace the element
$\gamma_1$ by its inverse $-\gamma_1$, then $G_1$ is the graph $X(s,
t, s-r)$.\end{proof}


In what follows, we show that for $s, t\geq 1$ there exists a Cayley
multigraph on a suitable abelian group that satisfies Proposition
\ref{pro_cayley}, i.e., for every $s, t\geq 1$ the graph $X(s, t,
r)$ can be represented as a Cayley multigraph. The proof is
particularly easy when $t=1$; $r=0$; or $s=2$. For these cases, the
following holds.

\begin{proposition}\label{pro_r0}
The graph $X(s, 1, r)$, with $s\geq 1$, $0\leq r\leq s-1$, is the
circulant multigraph $Cir(s$; $\pm 1, \pm r)$. The graph $X(s, t,
0)$, with $s, t\geq 1$, is the Cayley multigraph $Cay(\mathbb
Z_s\times\mathbb Z_t$, $\{\pm(1, 0)$, $\pm(0,1)\})$. The graph $X(2,
t, 1)$, with $t\geq 1$, is the circulant multigraph $Cir(2t; \pm t,
\pm 1)$.
\end{proposition}

\begin{proof}
For the graph $X(s, 1, r)$ we apply Proposition \ref{pro_cayley}
with $\Gamma=\mathbb Z_{st}$, $\gamma_1=1$, $\gamma_2=r$. For the
graph $X(s, t, 0)$ we apply Proposition \ref{pro_cayley} with
$\Gamma=\mathbb Z_s\times\mathbb Z_t$, $\gamma_1=(1,0)$,
$\gamma_2=(0, 1)$. For the graph $X(2, t, 1)$ we apply Proposition
\ref{pro_cayley} with $\Gamma=\mathbb Z_{2t}$, $\gamma_1=t$,
$\gamma_2=1$.
\end{proof}


The following lemmas concern the graph $X(s, t, r)$ with $s\geq 3$,
$t\geq 2$ and $0<r\leq s-1$. They will be used in the proof of
Proposition \ref{pro_circulant}.

\begin{lemma}\label{lemma_coprime}
Let $a>1$ be an integer and let $b\geq 1$ be a divisor of $a$. Let
$\{[c]_b : 0\leq c\leq b-1\}$ be the residue classes modulo $b$.
Every equivalence class $[c]_b$ whose representative $c$ is coprime
with $b$ contains at least one integer $h$, $1\leq h\leq a-1$, such
that $\gcd(a, h)=1$.
\end{lemma}

\begin{proof} The assertion is true if $b=a$ (we set $h=c$). We
consider $b<a$. Let $[c]_{b}$ be an equivalence class modulo $b$
with $1\leq c\leq b-1$ and $\gcd(c, b)=1$. If $c$ is coprime with
$a$, then we set $h=c$ and the assertion follows. We consider the
case $\gcd(c, a)\neq 1$. We denote by $\mathcal P$ the set of
distinct prime numbers dividing $a$. We denote by $\mathcal P_{b}$
(respectively, by $\mathcal P_c$) the subset of $\mathcal P$
containing the prime numbers dividing $b$ (respectively, $c$). Since
$b$ is a divisor of $a$ (respectively, $\gcd(c, a)\neq 1$), the set
$\mathcal P_{b}$ (respectively, $\mathcal P_c$) is non-empty. Since
$c$ and $b$ are coprime, the subsets $\mathcal P_{b}$, $\mathcal
P_c$ are disjoint. We set $\mathcal P'=\mathcal P\smallsetminus
(\mathcal P_{b}\cup\mathcal P_c)$. The set $\mathcal P'$ might be
empty. We denote by $\omega$ the product of the prime numbers in
$\mathcal P'$ (if $\mathcal P'$ is empty, then we set $\omega=1$)
and consider the integer $h=c+\omega b\in [c]_b$. We show that $h<
a$. Note that $\omega\leq a/(2b)$. More specifically, $(a/b)\geq
(\prod_{p\in\mathcal P_c}p)\cdot\omega\geq 2\omega$, whence
$\omega\leq a/(2b)$. Hence $h=c+\omega b\leq c+ (a/2)<a$, since
$c<b$ and $b\leq (a/2)$. One can easily verify that $\gcd(h,
a)=\gcd(c+\omega b, a)=1$, since no prime number in $\mathcal
P=\mathcal P'\cup\mathcal P_b\cup\mathcal P_c$ can divide $c+\omega
b$. The assertion follows.\end{proof}

\begin{lemma}\label{lemma_coprime2}
Let $s\geq 3$, $t\geq 2$ and $\mathbb Z_{st/d_1}$ be the cyclic
group of integers modulo $st/d_1$, where $d_1\geq 1$ is a divisor of
$d=\gcd(s, t)$. Let $\langle t/d_1\rangle$ be the cyclic subgroup of
$\mathbb Z_{st/d_1}$ generated by the integer $t/d_1$ and let
$x+\langle t/d_1\rangle$, $y+\langle t/d_1\rangle$ be left cosets of
$\langle t/d_1\rangle$ in $\mathbb Z_{st/d_1}$. If $x$, $y\in
\mathbb Z_{st/d_1}$ are congruent modulo $d/d_1$, then $t(x+\langle
t/d_1\rangle)=\{tx+\mu t^2/d_1 :0\leq\mu\leq s-1\}$ and $t(y+\langle
t/d_1\rangle)=\{ty+\mu' t^2/d_1 :0\leq\mu'\leq s-1\}$ represent the
same subset of $\mathbb Z_{st/d_1}$.
\end{lemma}

\begin{proof} Set $x=y+\lambda d/d_1$ with $\lambda\in\mathbb Z$ and
$t=sm'+m$ with $m'\in\mathbb Z$ and $0\leq m\leq s-1$. Since
$\gcd(s, t)=d$, then also $\gcd(s, m)=d$. Hence the integers
$dt/d_1,  mt/d_1\in\mathbb Z_{st/d_1}$ generate the same cyclic
subgroup of $\langle t/d_1\rangle$ of order $s/d$. Since
$t^2/d_1=(sm'+m)t/d_1\equiv mt/d_1\pmod{st/d_1}$, each set
$t(x+\langle t/d_1\rangle)$, $t(y+\langle t/d_1\rangle)$ consists of
exactly $s/d$ distinct elements of $\mathbb Z_{st/d_1}$, namely,
$t(x+\langle t/d_1\rangle)=\{tx+\mu mt/d_1: 1\leq \mu\leq s/d\}$,
$t(y+\langle t/d_1\rangle)=\{ty+\mu' mt/d_1: 1\leq \mu'\leq s/d\}$.
Therefore, to prove that $t(x+\langle t/d_1\rangle)=t(y+\langle
t/d_1\rangle)$, it suffices to show that every element of
$t(x+\langle t/d_1\rangle)$ is contained in $t(y+\langle
t/d_1\rangle)$. Consider $tx+\mu mt/d_1\in t(x+\langle
t/d_1\rangle)$. Since $x=y+\lambda d/d_1$, we can write $tx+\mu
mt/d_1=t(y+\lambda d/d_1)+\mu mt/d_1$, whence $tx+\mu
mt/d_1=ty+\lambda dt/d_1+\mu mt/d_1$. Since $\langle dt/d_1\rangle
=\langle mt/d_1\rangle$, we can set $\lambda dt/d_1+\mu mt/d_1\equiv
\mu' mt/d_1\pmod {st/d_1}$, with $0\leq \mu'\leq s/d$. Hence $tx+\mu
mt/d_1\equiv ty+\mu' mt/d_1\pmod{st/d_1}$, that is, $tx+\mu
mt/d_1\in t(y+\langle t/d_1\rangle)$. The assertion
follows.\end{proof}

\begin{proposition}\label{pro_circulant}
Let $s\geq 3$, $t\geq 2$, $0<r\leq s-1$, with $\gcd(s, t, r)=d_1$.
The cyclic group $\mathbb Z_{st/d_1}$ of integers modulo $st/d_1$
contains an integer $k$ such that $\gcd(k, t)=1$ and $k\equiv
r/d_1\pmod{s/d_1}$. The graph $X(s, t, r)$ can be represented as the
Cayley graph $Cay(\mathbb Z_{st/d1}\times \mathbb Z_{d_1}, \{\pm
(t/d_1, 1), \pm(k, 0)\})$. In particular, if $d_1=1$ then $X(s, t,
r)$ can be represented as the circulant graph $Cir(st; \pm t, \pm
k)$.
\end{proposition}

\begin{proof} Set $d=\gcd(s,t)$. Since $\gcd(s,t,r)=d_1$, the
integer $r/d_1$ is coprime with $\gcd(s/d_1, t/d_1)=d/d_1$. Hence,
$r/d_1$ belongs to an equivalence class $[c]_{d/d_1}$ whose
representative is coprime with $d/d_1$. By Lemma
\ref{lemma_coprime}, the class $[c]_{d/d_1}$ contains an integer
$h$, $1\leq h<t$, such that $\gcd(h,t)=1$. Consider the cyclic group
$\mathbb Z_{st/d_1}$. Since $r/d_1<s$, $h<t$, the integer $r/d_1$
and $h$ belong to $\mathbb Z_{st/d_1}$. Hence we can apply Lemma
\ref{lemma_coprime2} with $x=r/d_1$, $y=h$ and find that
$t(r/d_1+\langle t/d_1\rangle)=t(h+\langle t/d_1\rangle)$, i.e.,
there exists an integer $k\in h+\langle t/d_1\rangle$ such that
$tk\equiv rt/d_1\pmod{st/d_1}$. The integer $k$ is coprime with $t$,
since $\gcd(h, t)=1$. The assertion follows from Proposition
\ref{pro_cayley}, by setting $\Gamma=\mathbb Z_{st/d_1}\times\mathbb
Z_{d_1}$, $\gamma_1=(t/d_1, 1)$, $\gamma_2=(k,0)$. Note: if $d_1=1$,
then $\Gamma$ is the cyclic group of order $st$, $\gamma_1=t$,
$\gamma_2=k$ and $X(s, t, r)$ is the circulant graph $Cir(st; \pm t,
\pm k)$.\end{proof}


The result that follows is based on a well-known consequence of the
Chinese Remainder Theorem. More specifically, it is known that if
$a$, $b$, and $n$ are positive integers, with $\gcd(a, n)=c\geq 1$,
then the equation $ax\equiv b\pmod n$ admits a solution if and only
if $c$ is a divisor of $b$ and in this case $x\equiv
(a/c)^{-1}(b/c)\pmod{n/c}$ is a solution to the equation. The
following holds.

\begin{proposition}\label{pro_from_chinese_remainder}
Let $s, t\geq 1$, $0\leq r\leq s-1$ and $d_1=\gcd(s, t, r)$. If
$r\neq 0$, then there exists an integer $k$, $0<k<st/d_1$, such that
$\gcd(k, t)=1$ and $k\equiv r/d_1\pmod{s/d_1}$. The graph $X(s, t,
r)$, with $r\neq 0$, is isomorphic to the graph $X(st/\gcd(s, r)$,
$\gcd(s, r)$, $r')$, where $r'\equiv\pm t(kd_1/\gcd(s,
r))^{-1}\pmod{st/\gcd(s,r)}$. The graph $X(s, t, 0)$ is isomorphic
to the graph $X(t, s, 0)$.
\end{proposition}

\begin{proof} We prove the assertion for $s\geq 3$, $t\geq 2$ and
$0< r\leq s-1$. The existence of the integer $k$ follows from
Proposition \ref{pro_circulant}. By the same proposition, we can
represent the graph $X(s, t, r)$ as $Cay(\mathbb Z_{st/d1}\times
\mathbb Z_{d_1}, \{\pm (t/d_1, 1), \pm(k, 0)\})$. We apply
Proposition \ref{pro_cayley} by setting $\Gamma=\mathbb
Z_{st/d1}\times \mathbb Z_{d_1}$, $\gamma_1=(k, 0)$ and
$\gamma_2=(t/d_1, 1)$. Note that $\gcd(st/d_1, k)$$=\gcd(s/d_1,
k)$$=\gcd(s, r)/d_1$, as $k$ is coprime with $t$ and $k\equiv
r/d_1\pmod{s/d_1}$. Whence the element $(k, 0)$ has order
$s'=st/(d_1\gcd(st/d_1,k))=st/\gcd(s, r)$ and $t'=|\Gamma/\langle
(k, 0)\rangle|=\gcd(s, r)$. By Proposition \ref{pro_cayley},
$\gcd(s, r)(t/d_1, 1)=r'(k, 0)$ for some integer $r'$, $1\leq r'\leq
st/\gcd(s, r)$. The integer $r'$ is a solution to the equation
$\gcd(s, r)(t/d_1)\equiv r'k\pmod{st/d_1}$. By the Chinese Remainder
Theorem, $r'\equiv t(kd_1/\gcd(s, r))^{-1}\pmod{st/\gcd(s, r)}$. An
easy calculation shows that $s'-r'\equiv -t(kd_1/\gcd(s,
r))^{-1}$$\pmod{st/\gcd(s,r)}$. It is straightforward to see that
$X(s, t, r)$ and $X(s', t', r')$, $X(s'$$, t'$$, s'-r')$ are
isomorphic. Hence the assertion follows. For the remaining values of
$s$, $t$, $r$, we represent the graph $X(s, t, r)$ as the Cayley
multigraph in Proposition \ref{pro_r0} and use Proposition
\ref{pro_cayley}. Note: if $r\neq 0$, then $k=r$; if $r=0$, then set
$\gamma_1=(0, 1)$, $\gamma_2=(1, 0)$ and apply Proposition
\ref{pro_cayley}.\end{proof}

\subsection{Fundamental $2$-factorization of $X(s, t,
r)$.}\label{sec_fundamental}

From the definition of $X(s, t, r)$ one can see that the horizontal
edges form a $2$-factor (the \emph{red $2$-factor}) whose
complementary $2$-factor in $X(s, t, r)$ is given by the vertical
and diagonal edges (the \emph{blue $2$-factor}). We say that the red
and blue $2$-factor constitute the \emph{fundamental
$2$-factorization} of $X(s, t, r)$. A graph $X(s, t, r)$ can be
represented as a Cayley multigraph $Cay(\Gamma,$ $\{\pm\gamma_1,
\pm\gamma_2\})$, where $\Gamma$ and $\{\pm\gamma_1, \pm\gamma_2\}$
can be defined as in Proposition \ref{pro_r0} or
\ref{pro_circulant}. From the proof of the propositions, one can see
that the set of horizontal edges of $X(s, t, r)$ is the set $\{[x,
x\pm\gamma_1] : x\in\Gamma\}$, the set of vertical and diagonal
edges is the set $\{[x, x\pm\gamma_2] : x\in\Gamma\}$. The edges in
$\{[x, x\pm\gamma_1] : x\in\Gamma\}$ will be called the
\emph{$\gamma_1$-edges}, the edges in the set $\{[x, x\pm\gamma_2] :
x\in\Gamma\}$ will be called the \emph{$\gamma_2$-edges}. The
following result holds.

\begin{proposition}\label{pro_red_blu_2factor}
The red $2$-factor of $X(s, t, r)$ has exactly $t$ cycles of length
$s$ consisting of $\gamma_1$-edges. The blue $2$-factor of $X(s, t,
r)$ has exactly $\gcd(s, r)$ cycles of length $st/\gcd(s, r)$
consisting of $\gamma_2$-edges.
\end{proposition}

\begin{proof} It is straightforward to see that the red $2$-factor
has $t$ horizontal cycles of length $s$ (if $s=1$, then each cycle
is a loop; if $s=2$, then each cycle is a dipole with $2$ parallel
edges). By the previous remarks, each cycle consists of
$\gamma_1$-edges. The blue $2$-factor of $X(s, t, r)$ corresponds to
the red $2$-factor of the graph $X(st/\gcd(s, r), \gcd(s, r), r')$
in Proposition \ref{pro_from_chinese_remainder}. Hence it has
$\gcd(s, r)$ cycles of length $st/\gcd(s, r)$ consisting of
$\gamma_2$-edges.\end{proof}

\subsection{Isomorphisms between $X(s, t, r)$
graphs.}\label{sec_isoX}

We wonder whether two graphs $X(s, t, r)$ and $X(s', t', r')$ are
isomorphic. Our question is connected to the following well-known
problem \cite{DeFaMa}: given two isomorphic Cayley multigraphs
$Cay(\Gamma, S)$, $Cay(\Gamma', S')$, determine whether the groups
$\Gamma$, $\Gamma'$ are necessarily A-isomorphic, that is, there
exists an isomorphism $\tau$ between $\Gamma$ and $\Gamma'$ that
sends the set $S$ onto the set $S'$ ($\tau$ is called an
\emph{A-isomorphism} and is denoted by $\tau:(\Gamma, S)\to
(\Gamma', S')$). Ad\'am \cite{Adam} considered this problem for
circulant graphs and formulated a well-known conjecture which was
disproved in \cite{ElTu}. He conjectured that two circulant graphs
$Cir(n; S)$, $Cir(n; S')$ are isomorphic if and only if there exists
an integer $m'\in\mathbb Z_n$, $\gcd(m', n)=1$, such that $S'=\{m'x:
x\in S\}$. Even though the conjecture was disproved, there are some
circulant graphs that verify it (see for instance \cite{Mu}). In
\cite{DeFaMa} the problem is studied for Cayley multigraphs of
degree $4$ which are associated to abelian groups. The results in
\cite{DeFaMa} are described in terms of Ad\'am isomorphisms. We
recall that an \emph{Ad\'am isomorphism} between the graphs
$Cay(\Gamma, S)$, $Cay(\Gamma, S')$ is an isomorphism of type
$f_{\sigma'}\cdot\tau$, where $\tau:(\Gamma, S)\to (\Gamma', S')$ is
an A-isomorphism  and $f_{\sigma'}$ is the automorphism of the graph
$Cay(\Gamma', S')$ defined by $f_{\sigma'}(x')=\sigma'+x'$ for every
vertex $x'\in\Gamma'$. Since the graphs $X(s, t, r)$ can be
represented as Cayley multigraphs, we can extend the notion of
Ad\'am isomorphism to the graphs $X(s, t, r)$. We will say that the
graphs $X(s, t, r)$, $X(s', t', r')$ are \emph{Ad\'am isomorphic} if
the corresponding Cayley multigraphs $Cay(\Gamma, \{\pm\gamma_1,
\pm\gamma_2\})$, $Cay(\Gamma', \{\pm\gamma'_1, \pm\gamma'_2\})$,
respectively, are Ad\'am isomorphic ($Cay(\Gamma, \{\pm\gamma_1,
\pm\gamma_2\})$, $Cay(\Gamma', \{\pm\gamma'_1, \pm\gamma'_2\})$ are
described in Proposition \ref{pro_r0} or \ref{pro_circulant}). The
following statements hold.

\begin{proposition}\label{pro_adam_iso1}
Every Ad\'am isomorphism between the graphs $X(s, t, r)$, $X(s', t',
r')$ sends the fundamental $2$-factorization of $X(s, t, r)$ onto
the fundamental $2$-factorization of $X(s', t', r')$.
\end{proposition}

\begin{proof} An Ad\'am isomorphism $f_{\sigma'}\cdot\tau$ between
the graphs $Cay(\Gamma$, $\{\pm\gamma_1$, $\pm\gamma_2\})$,
$Cay(\Gamma'$, $\{\pm\gamma'_1$, $\pm\gamma'_2\})$ sends a
$\gamma_i$-edge, $i=1, 2$, of $Cay(\Gamma, \{\pm\gamma_1,
\pm\gamma_2\})$ onto a $\tau(\gamma_i)$-edge of $Cay(\Gamma'$,
$\{\pm\gamma'_1$, $\pm\gamma'_2\})$, where
$\tau(\gamma_i)\in\{\pm\gamma'_1, \pm\gamma'_2\}$. Since Proposition
\ref{pro_red_blu_2factor} holds, every Ad\'am isomorphism sends the
red (respectively, the blue) $2$-factor of $X(s, t, r)$ onto the red
(respectively, the blue) $2$-factor of $X(s', t', r')$ or vice
versa.\end{proof}

\begin{proposition}\label{pro_adam_iso2}
Let $s, t\geq 1$, $0\leq r\leq s-1$ and $\gcd(s, t, r)=d_1$. If
$r\neq 0$ then there exists an integer $k$, $0<k<st/d_1$, such that
$\gcd(k, t)=1$ and $k\equiv r/d_1$$\pmod{s/d_1}$. The graphs $X(s,
t, r)$, with $r\neq 0$, and $X(s', t', r')$ are Ad\'am isomorphic if
and only if $s'=s$, $t'=t$, $r'=s-r$ or $s'=st/\gcd(s, r)$,
$t'=\gcd(s, r)$ and $r'\equiv\pm t(kd_1/\gcd(s,
r))^{-1}$$\pmod{st/\gcd(s, r)}$. The graphs  $X(s, t, 0)$, and
$X(s', t', r')$ are Ad\'am isomorphic if and only if $s'=t$, $t'=s$
and $r'\equiv 0\pmod t$.
\end{proposition}

\begin{proof} We prove the assertion for $s\geq 3$, $t\geq 2$ and
$r\neq 0$. The graph $X(s, t, r)$ is the Cayley graph $Cay(\mathbb
Z_{st/d_1}\times\mathbb Z_{d_1}, \{\pm (t/d_1, 1), \pm (k, 0)\})$,
since Proposition \ref{pro_circulant} holds. By Proposition
\ref{pro_r0} or \ref{pro_circulant}, we can represent
$X(s',$$t'$,$r')$ as the Cayley multigraph $Cay(\Gamma',
\{\pm\gamma'_1, \pm\gamma'_2\})$. The graphs $X(s, t, r)$, $X(s',
t', r')$ are Ad\'am isomorphic if and only if there exists an
isomorphism $\tau$ between the groups $\mathbb
Z_{st/d_1}\times\mathbb Z_{d_1}$ and $\Gamma'$ that sends the set
$\{\pm (t/d_1, 1), \pm (k, 0)\}$ onto the set $\{\pm\gamma'_1,
\pm\gamma'_2\}$. Without loss of generality, we can set
$\{\pm\gamma'_1\}$$=\{\pm\tau((t/d_1, 1))\}$ and
$\{\pm\gamma'_2\}$$=\{\pm\tau((k, 0))\}$. By the existence of $\tau$
we can identify the group $\Gamma'$ with the group $\mathbb
Z_{st/d_1}\times\mathbb Z_{d_1}$. Hence $\gamma'_1$ and $\gamma'_2$
are elements of $Z_{st/d_1}\times\mathbb Z_{d_1}$ of order $s$ and
$st/\gcd(s, r)$, respectively, since $(t/d_1, 1)$ and $(k, 0)$ have
order $s$ and $st/\gcd(s, r)$, respectively (see the proof of
Proposition \ref{pro_circulant} and
\ref{pro_from_chinese_remainder}). It is an easy matter to prove
that an element $(a, b)\in\mathbb Z_{st/d_1}\times\mathbb Z_{d_1}$
has order $o(a)\cdot o(b)/\gcd(o(a), o(b))=st/(d_1\gcd(st/d_1, a))$,
since $d_1$ is a divisor of $s$ and $t$. Hence $(a, b)$ has order
$s$ if and only if $\gcd(st/d_1, a)=t/d_1$, i.e., $(a, b)=(m't/d_1,
b)$ where $m'\in\mathbb Z_{st/d_1}$, $\gcd(m', s)=1$, $b$ is an
arbitrary element of $\mathbb Z_{d_1}$. The element $(a, b)$ has
order $st/\gcd(s, r)$ if and only if $\gcd(st/d_1, a)=\gcd(s,
r)/d_1=\gcd(s/d_1, k)$, since $k$ is coprime with $t$ and $k\equiv
r/d_1\pmod{s/d_1}$. Hence $Cay(\Gamma', \{\pm\gamma'_1,
\pm\gamma'_2\})$ is a graph of type $Cay(\mathbb
Z_{st/d_1}\times\mathbb Z_{d_1}, \{\pm (m't/d_1, b), \pm(a, b')\})$,
where $\gcd(m', s)=1$, $\gcd(st/d_1, a)=\gcd(s/d_1, k)$, $b$ and
$b'$ are suitable elements of $\mathbb Z_{d_1}$. Note that $a$ is
coprime with $t$ and the relation $ta\equiv rm't/d_1\pmod{st/d_1}$
holds, since $\tau$ is an isomorphism and $tk\equiv
rt/d_1\pmod{st/d_1}$. If we apply Proposition \ref{pro_cayley} to
the graph $G_1=$$Cay(\mathbb Z_{st/d_1}\times\mathbb Z_{d_1}$,
$\{\pm (m't/d_1$$, b)$$, \pm(a$$, b')\})$ by setting
$\gamma_1=(m't/d_1, b)$ (or $\gamma_1=-(m't/d_1, b)$), then $G_1$
can be represented as the graph $X(s, t, r)$ or $X(s, t, s-r)$. The
graph $X(s, t, r)$ is isomorphic to the graph $G_2=X(s', t', r')$,
where $s'=st/\gcd(s, r)$, $t'=\gcd(s, r)$, $r'\equiv\pm
t(kd_1/\gcd(s, r))^{-1}$$\pmod{st/\gcd(s, r)}$, since Proposition
\ref{pro_from_chinese_remainder} holds. Hence $G_1$ is isomorphic to
$G_2$. The isomorphism between $G_1$ and $G_2$ can be obtained also
by applying Proposition \ref{pro_from_chinese_remainder} to the
graph $G_1$. For the remaining values of $s$, $t$, $r$ we represent
the graph $X(s, t, r)$ as the Cayley multigraph in Proposition
\ref{pro_r0} and apply the previous method.\end{proof}



The results that follow are based on the following theorem of
\cite{DeFaMa}.

\begin{theorem}[(\cite{DeFaMa})]\label{th_delrome1}
Any two finite isomorphic connected undirected Cayley multigraphs of
degree $4$ coming from abelian groups are Ad\'am isomorphic, unless
they are obtained with the groups and sets $\mathbb Z_{4n}$, $\{\pm
1, \pm (2n-1)\}$ and $\mathbb Z_{2n}\times\mathbb Z_2$, $\{\pm
(1,0), \pm (1, 1)\}$.
\end{theorem}

By Theorem \ref{th_delrome1} the existence of an isomorphism between
two Cayley multigraphs of degree $4$, that are associated to abelian
groups, implies the existence of an Ad\'am isomorphism, unless they
are the graphs $Cir(4n; \pm 1, \pm (2n-1))$ and $Cay(\mathbb
Z_{2n}\times\mathbb Z_2$, $\{\pm (1,0)$, $\pm (1,1)\})$. The
following statement are a consequence of Theorem \ref{th_delrome1}.

\begin{corollary}\label{pro_cir_4n}
The graphs $X(4n, 1, 2n-1)$ and $X(s', t', r')$ are isomorphic if
and only if $s'=4n$, $t'=1$, $r'=2n+1$ or $s'=2n$, $t'=2$,
$r'\in\{2, 2n-2\}$. Moreover, there is no isomorphism between $X(4n,
1, 2n-1)$ and $X(2n, 2, 2)$ that sends the fundamental
$2$-factorization of $X(4n, 1, 2n-1)$ onto the fundamental
$2$-factorization of $X(2n, 2, 2)$.
\end{corollary}

\begin{proof} The graph $X(4n, 1, 2n-1)$ is the graph $Cir(4n; \pm
1,\pm (2n-1))$ (see Proposition \ref{pro_r0}). By Theorem
\ref{th_delrome1}, the graphs $X(4n, 1, 2n-1)$ and $X(s', t', r')$
could be Ad\'am isomorphic or not. If they are Ad\'am isomorphic,
then $s'=4n$, $t'=1$, $r'=2n+1$, since Proposition
\ref{pro_adam_iso2} holds. If they are not Ad\'am isomorphic, then
$X(s', t', r')$ is the graph $Cay(\mathbb Z_{2n}\times\mathbb Z_2$,
$\{\pm (1,0)$, $\pm (1,1)\})$ (see Theorem \ref{th_delrome1}). Hence
$s'=2n$, $t'=2$, $r'\in\{2, 2n-2\}$ (see Proposition
\ref{pro_circulant} and \ref{pro_adam_iso2}). The fundamental
$2$-factorization of $X(4n, 1, 2n-1)$ consists of two Hamiltonian
cycles, whereas the fundamental $2$-factorization of $X(2n, 2, 2)$
consists of two $2$-factors whose connected components are two
$2n$-cycles (see Proposition \ref{pro_red_blu_2factor}). Therefore
no isomorphism between $X(4n, 1, 2n-1)$ and $X(2n, 2, 2)$ can send
the fundamental $2$-factorization of $X(4n, 1, 2n-1)$ onto the
fundamental $2$-factorization of $X(2n, 2, 2)$.\end{proof}

\begin{proposition}\label{pro_iso_X}
Let $X(s, t, r)$, $X(s', t', r')$ be non-isomorphic to $X(4n, 1,
2n-1)$, $X(2n, 2, 2)$. Then $X(s, t, r)$ and $X(s', t', r')$ are
isomorphic if and only if they are Ad\'am isomorphic, i.e., if and
only if  the parameters $s'$, $t'$, $r'$ satisfy Proposition
\ref{pro_adam_iso2}.
\end{proposition}

\begin{proof} The assertion follows from  Theorem \ref{th_delrome1},
and Proposition \ref{pro_adam_iso2}.\end{proof}

\section{Special cubic graphs arising from $X(s, t, r)$ graphs.}\label{sec_generalizing_Igraphs}

When we consider graphs $X(s,t,r)$ we assume we are given a
fundamental 2-factorization. This, in turn, implies we may turn the
graph $X(s,t,r)$ into a cubic one by appropriately splitting each
vertex. We note in passing that the operation of vertex-splitting
and its converse were successfully used in a different context in
\cite{PSV13}.

There are two complementary possibilities. Either $X(s,t,r)$ arises
from an $I$-graph or not. We consider each case separately.

\subsection{$I$-graphs arising from $X(s,t,r)$.}\label{sec_i_graphs}

In Theorem \ref{th_cubic_quartic} we remarked that any special cubic
graph with a blue and red $2$-factorization gives rise to the
associated quartic graph with a blue and red $2$-factorization. In
Lemma \ref{pro_Igraph_special}, we showed that a proper $I$-graph
$I(n, p, q)$ is special and gives rise to the associated circulant
graph $Q(n, p, q)$. The following holds.

\begin{lemma}\label{pro_circulant2}
The circulant graph $Cir(n; p, q) = Q(n,p,q)$ arising from a
connected $I$-graph $I(n, p, q)$ by contracting the spokes is the
graph $X(s, t, r)$ with $t=\gcd(n, q)$, $s=n/t\geq 3$ and
$r\equiv\pm p(q/t)^{-1}$$\pmod s$.
\end{lemma}

\begin{proof} The result follows from Proposition \ref{pro_cayley}
by setting $\Gamma=\mathbb Z_{n}$, $\gamma_1=q$, $\gamma_2=p$.
Whence $tp=rq$ for some integer $r$,  $0\leq r\leq s-1$, i.e., $r$
is a solution to the equation $r(q/t)\equiv p\pmod s$. By the
Chinese Remainder Theorem, $r\equiv p(q/t)^{-1}$$\pmod
s$.\end{proof}

\begin{theorem}\label{pro_gcd_1}
The graph $X(s, t, r)$ arises from a connected $I$-graph by
contracting the spokes if and only if $\gcd(s, t, r)=1$ and $(t,
r)$$\neq (2, 0)$ for odd values of $s$. In this case, the graph
$X(s, t, r)$ together with its fundamental $2$-factorization, is in
one-to-one correspondence with the $I$-graph $I(st, k, t)$, where
$0<k<st$, $\gcd(k, t)=1$ and $k\equiv r\pmod{s}$ (in particular,
$k=s$ if $r=0$). If at least one of the integers $k$, $t$, $\gcd(s,
r)$ is $1$, then $X(s, t, r)$ corresponds to a generalized Petersen
graph.
\end{theorem}

\begin{proof} Assume that $X(s, t, r)$ arises from the connected
$I$-graph $I(n, p, q)$ by contracting the spokes. By Lemma
\ref{pro_circulant2}, $t=\gcd(n, q)$, $s=n/t\geq 3$ and
$r(q/t)\equiv p\pmod{s}$. Whence $(t, r)\neq (2, 0)$ if $s$ is odd,
otherwise $p=0$ (which is not possible). We show that $\gcd(s, t,
r)=1$. Suppose, on the contrary, that $\gcd(s, t, r)=d_1\neq 1$,
then $d_1$ is a divisor of $\gcd(t, p)$ since $r(q/t)\equiv p\pmod
s$. That yields a contradiction, since $\gcd(t, p)=1$ (see
Proposition \ref{pro8}). Hence $\gcd(s, t, r)=1$.

Assume that $\gcd(s, t, r)=1$. We show that $X(s, t, r)$ arises from
a connected $I$-graph by contracting the spokes. Since $\gcd(s, t,
r)=1$, the graph $X(s, t, r)$ can be represented as the circulant
graph $Cir(st; \pm t, \pm k)$, where $0<k<st$, $\gcd(t,k)=1$ and
$k\equiv r\pmod{s}$ (see Proposition \ref{pro_circulant}). If $r=0$,
then we can set $k=s$, since Proposition \ref{pro_r0} holds. The
graph $I(st, k, t)$ is connected and it gives rise to the graph
$X(s, t, r)$, since Lemma \ref{pro_circulant2} holds. By Theorem
\ref{th_cubic_quartic}, the graph $X(s, t, r)$, together with its
fundamental $2$-factorization, is in one-to-one correspondence with
the $I$-graph $I(st, k, t)$. If $k=1$ or $t=1$, then $X(s, t, r)$
corresponds to a generalized Petersen graph (see \cite{BPZ}). If
$\gcd(s, r)=1$ then $X(s, t, r)$ is isomorphic to $X(st, 1, r')$
(see Proposition \ref{pro_adam_iso2}. By the previous remarks, the
graph $X(st, 1, r')$ corresponds to a generalized Petersen graph.
The assertion follows.\end{proof}


It is straightforward to see that isomorphic $X(s, t, r)$ graphs
give rise to isomorphic $I$-graphs and also the converse is true. By
Corollary \ref{pro_cir_4n} and Proposition \ref{pro_iso_X}, the
circulant graphs $X(s, t, r)$, $X(s', t', r')$ are isomorphic if and
only if they are Ad\'am-isomorphic, i.e., there exists an
automorphism of the cyclic group of order $st=s't'$ that sends the
defining set of the circulant graph $X(s, t, r)$ onto the defining
set of the circulant graph $X(s', t', r')$. This fact is equivalent
to the results proved in \cite{HPZ} about the isomorphism between
$I$-graphs.

\subsection{Special Generalized $I$-graphs.}\label{sec_sgi_graphs}

In this section we consider the special cubic graphs that correspond
to the graphs $X(s, t, r)$ with $\gcd(s, t, r)\neq 1$, according to
the correspondence described in Theorem \ref{th_cubic_quartic}. By
Proposition \ref{pro_gcd_1}, these special cubic graphs do not
belong to the family of connected $I$-graphs. By Theorem
\ref{th_cubic_quartic} and Definition \ref{defX}, we can define a
family of special cubic graphs containing the family of connected
$I$-graphs as a subfamily. We call this family \emph{Special
Generalized $I$-graphs}. This family is not contained in the family
of $GI$-graphs \cite{MPZ}.

\smallskip

Let $s\geq 1$, $t\geq 1$ and $0\leq r\leq s-1$. We define a
\emph{Special Generalized $I$-graph} $SGI(st, s, t, r)$ as a cubic
graph of order $st$ with vertex-set $V=\{u_{i,j}, u'_{i,j}: 0\leq
i\leq t-1, 0\leq j\leq s-1\}$ and edge-set $E=\{[u_{i,j}, u_{i,
j+1}]$, $[u_{i,j}, u'_{i,j}]:$ $0\leq i\leq t-1, 0\leq j\leq
s-1\}$$\cup\{[u'_{i,j}, u'_{i+1,j}]: 0\leq i\leq t-2, 0\leq j\leq
s-1\}$$\cup\{[u'_{t-1,j}, u'_{0,j+r}]: 0\leq j\leq s-1\}$ (the
addition $j+1$ and $j+r$ are considered modulo $s$). For $s=1$ or
$(t, r)=(1, 0)$, a special generalized $I$-graph has loops. For
$s=2$ or $(t, r)=(2, 0)$, it has multiple edges. For the other
values of $s$, $t$, $r$, it is a simple cubic graph. We say that a
vertex $u_{i,j}$ (respectively, $u'_{ij}$) is an \emph{outer vertex}
(respectively, an \emph{inner vertex}). We say that an edge
$[u_{i,j}, u_{i, j+1}]$ (respectively, $[u'_{i,j}, u'_{i+1,j}]$) is
an \emph{outer edge} (respectively, an \emph{inner edge}). We say
that an edge $[u_{i,j}, u'_{i,j}]$ is a \emph{spoke}. The spokes
constitute the special $1$-factor. The graph arising from $SGI(st,
s, t, r)$ by contracting the spokes is the graph $X(s, t, r)$. The
horizontal edges of $X(s, t, r)$ correspond to the outer edges of
$SGI(st, s, t, r)$, vertical and diagonal edges of $X(s, t, r)$
correspond to the inner edges of $SGI(st, s, t, r)$. A
generalization of the proof of Proposition \ref{pro_gcd_1} gives the
following statement.

\begin{proposition}\label{pro_gcd_1_gen}
Let $s\geq 1$, $t\geq 1$, $0\leq r\leq s-1$ and $d_1=\gcd(s, t, r)$.
The graph $X(s, t, r)$, together with its fundamental
$2$-factorization, is in one-to-one correspondence with the graph
$SGI(st, s, t, k)$ where $k=s$ if $r=0$, otherwise $0< k<st/d_1$,
$\gcd(k, t)=1$ and $k\equiv r/d_1\pmod{s/d_1}$.
\end{proposition}

By Corollary \ref{pro_cir_4n}, the graphs $X(4n, 1, 2n-1)$ and
$X(2n, 2, 2)$ are isomorphic, but no isomorphism between them sends
the fundamental $2$-factorization of $X(4n, 1, 2n-1)$ onto the
fundamental $2$-factorization of $X(2n, 2, 2)$. This fact means that
the application of Theorem \ref{th_cubic_quartic} to the graphs
$X(4n, 1, 2n-1)$ and $X(2n, 2, 2)$ yields non-isomorphic special
cubic graphs. As a matter of fact, Proposition \ref{pro_gcd_1} says
that the graph $X(4n, 1, 2n-1)$ is in one-to-one correspondence with
a connected $I$-graph, whereas $X(2n, 2, 2)$ does not correspond to
any $I$-graph. For instance, for $n = 2$ the graph $X(8,1,3)$ is
associated with the M\"{o}bius-Kantor graph of girth 6,
\cite{MP00,PS13}, while $X(4,2,2)$ arises from a graph of girth 4,
see Figure \ref{fig_nonMoebiusKantor_graphr}.

\begin{figure}[htbp]
\centering
\begin{minipage}{4cm}
\includegraphics[width=4cm]{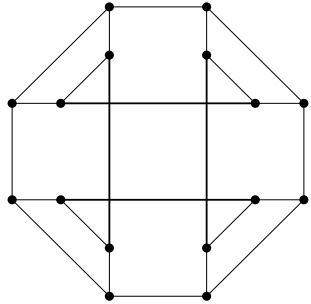}
\end{minipage}
\caption{The cubic split $X(4, 2, 2)$ graph is $SGI(8, 4, 2, 2)$.
The thick edges represent the special
1-factor.}\label{fig_nonMoebiusKantor_graphr}
\end{figure}

\section{Good Eulerian tours in $X(s, t, r)$ graphs.}\label{sec_good_tour_graph_Xstr}

In this section we construct good Eulerian subgraphs of $X(s, t,
r)$. For each $X(s, t, r)$ we denote by $W(s, t, r)$ the constructed
good Eulerian subgraph. By Proposition \ref{pro_admissible}, a
spanning Eulerian sub-quartic subgraph $W$ of $X(s, t, r)$ is
admissible if and only if at each $2$-valent vertex exactly one edge
is horizontal. We consider $X(s, t, r)$ being embedded into torus
with quadrilateral faces. Hence any of its subgraphs may be viewed
embedded in the same surface. A tour in $W$ may be regarded as as a
\emph{straight-ahead walk} (or SAW) on the surface \cite{PTZ}. A
good Eulerian tour of $W$ is an Eulerian SAW that uses only allowed
transitions, i.e. the tour cannot switch from an horizontal to a
vertical (or diagonal) edge when it visits a $4$-valent vertex of
$W$. For instance, the graph $W$ in Figure
\ref{fig_horizontal_expansion}(a) is an admissible subgraph of $X(5,
4, 3)$; the tour $\mathcal E=(x^0_0, x^1_0, x^1_1$, $x^1_2, x^1_3,
x^1_4$, $x^0_4, x^3_1, x^2_1$, $x^1_1, x^0_1, x^0_2$,$ x^1_2, x^2_2,
x^3_2$, $x^3_1, x^3_0, x^2_0$,$ x^2_1, x^2_2, x^2_3$, $x^2_4, x^3_4,
x^3_3$, $x^2_3, x^1_3, x^0_3$, $x^0_4, x^0_1)$ is a good Eulerian
SAW of $W$; hence $W$ is a good Eulerian subgraph of $X(5, 4, 3)$.

If we delete the diagonal edges in $X(s, t, r)$, we obtain a
spanning subgraph that we denote by $X'(s, t, r)$. Clearly $X'(s, t,
r)$ is the cartesian product of a cycle $C_s$ with a path $P_t$
embedded in torus or cylinder. If we further delete an edge in $C_s$
we obtain a path $P_s$. We denote the cartesian product of $P_s$ and
$P_t$ by $X''(s, t, r)$ and obtain a spanning subgraph of $X'(s, t,
r)$ and $X(s, t, r)$. In order to simplify the constructions, we
will seek to find good Eulerian subgraphs in $X'(s, t, r)$ or in
$X''(s, t, r)$. In this case the resulting good Eulerian subgraph
will be denoted by $W'(s, t, r)$ and $W''(s, t, r)$, respectively.
This simplification makes sense, since neither $X'(s, t, r)$ nor
$X''(s, t, r)$ depend on the parameter $r$. Hence any Eulerian
subgraph $W'(s, t, r)$ or $W''(s, t, r)$ is good for any $r$.

\subsection{Method of construction.}\label{subsec_technical}

We give some lemmas that will be used in the construction of a good
Eulerian subgraph $W(s, t, r)$. Given a graph $X(s, t, r)$, for
every row index $i$, $0\leq i\leq t-1$, we denote by $V_i$ the set
of vertical edges $V_i=\{[x^i_j, x^{i+1}_j] : 0\leq j\leq s-1\}$.
For every column index $j$, $0\leq j\leq s-1$, we denote by $H_j$
the set of horizontal edges $H_j=\{[x^i_j, x^i_{j+1}] : 0\leq i\leq
t-1\}$.

Let $H$ be a subgraph of $X(s, t, r)$. We say that $H$ can be
\emph{expanded vertically (from row $i$)} if $|E(H)\cap V_i|=s-1$ or
$s-2>0$ (for $s=3$ we require $|E(H)\cap V_i|=2$). We say that $H$
can be \emph{expanded horizontally (from column $j$)} if $|E(H)\cap
H_j|=t-1$ or $t-2>0$ (for $t=3$ we require $|E(H)\cap H_j|=2$). The
following statements hold.

\begin{lemma}\label{lemma_vertical}
Let $W(s, t_1, r)$ be a good Eulerian subgraph that can be expanded
vertically. Then there exists a good Eulerian subgraph $W(s, t, r)$
for every $t\geq t_1$, $t\equiv t_1\pmod 2$.
\end{lemma}

\begin{proof} We use the graph $W_1=W(s, t_1, r)$ to construct a
good Eulerian subgraph $W(s, t, r)$. By the assumptions,
$|E(W_1)\cap V_i|=s-1$ or $s-2$ for some row index $i$, $0\leq i\leq
t-1$. By the symmetry properties of the graph $X(s, t_1, r)$, we can
cyclically permute its rows so that we can assume $0<i<t-1$. We
treat separately the cases $|E(W_1)\cap V_i|=s-1$ and $|E(W_1)\cap
V_i|=s-2$. Consider $|E(W_1)\cap V_i|=s-1$ and denote by $[x^i_a,
x^{i+1}_a]$ the vertical edge of $V_i$ which is missing in $W_1$. We
can cyclically permute the columns of $X(s, t_1, r)$ and assume
$a=0$. We subdivide every vertical edge $[x^i_j, x^{i+1}_j]$, with
$0<j\leq s-1$, by inserting two new vertices, namely, $y^i_j$ and
$y^{i+1}_j$ ($y^i_j$ is adjacent to $x^i_j$; $y^{i+1}_j$ is adjacent
to $x^{i+1}_j$); we delete the edge $[y^{i}_{s-1}, y^{i+1}_{s-1}]$
and in correspondence of the missing edge $[x^i_0, x^{i+1}_0]$ we
add two new vertices, namely, $y^i_0$ and $y^{i+1}_0$; we add the
edges $[y^i_j, y^{i}_{j+1}]$,  $[y^{i+1}_j, y^{i+1}_{j+1}]$ with
$0\leq j\leq s-2$. The resulting graph is a good Eulerian subgraph
$W(s, t_1+2, r)$. We can iterate the process and find a good
Eulerian subgraph $W(s, t, r)$ for every $t\geq t_1$, $t\equiv
t_1\pmod 2$. The case $|E(W_1)\cap V_i|=s-2$ can be treated
analogously to the case $|E(W_1)\cap V_i|=s-1$. As an example,
consider the graph $W''(6, 5, r)$ in Figure
\ref{fig_vertical_expansion}. It can be expanded vertically from row
$1$ and it yields a good Eulerian subgraph $W''(6, 7,
r)$.\end{proof}

\begin{figure}[htbp]
\centering
\begin{minipage}{9cm}
\includegraphics[width=9cm]{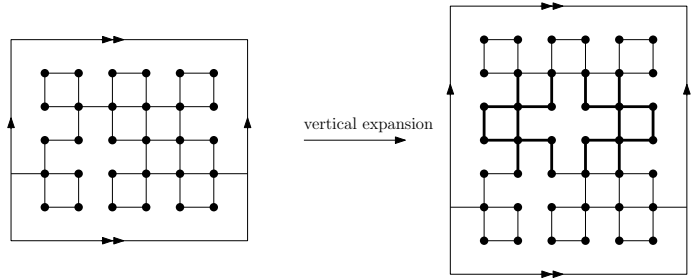}
\end{minipage}
\caption{A vertical expansion of the good Eulerian subgraph $W'(6,
5, r)$ yields a good Eulerian subgraph $W(6, 7, r)$.}
\label{fig_vertical_expansion}
\end{figure}

In the following lemma we consider horizontal expansions. In this
case we have to pay attention to the diagonal edges of $W(s, t, r)$,
if any exists. If $[x^{t-1}_j, x^0_{j+r}]$, where $j+r$ is
considered modulo $s$, is a diagonal edge of $W(s, t, r)$, then we
can assume $j<j+r$, since we can cyclically permute the columns of
$W(s, t, r)$. Therefore we can say that a diagonal edge $[x^{t-1}_j,
x^0_{j+r}]$ \emph{crosses column} $\ell$ if $j\leq\ell< j+r$.

\begin{lemma}\label{lemma_horizontal2}
Let $W(s_1, t, r_1)$ be a good Eulerian subgraph that can be
expanded horizontally from column $\ell$. If no diagonal edge of
$W(s_1, t, r_1)$ crosses column $\ell$, then there exists a good
Eulerian subgraph $W(s, t, r_1)$ for every $s\geq s_1$, $s\equiv
s_1\pmod 2$. If every diagonal edge crosses column $\ell$, then
there exists a good Eulerian subgraph $W(s_1-r+r_1, t, r)$ for every
$r_1\leq r\geq s_1-r_1$, $r\equiv r_1\pmod 2$.
\end{lemma}

\begin{proof} We apply the method described in Lemma
\ref{lemma_vertical} to the edges in $H_{\ell}$. If every diagonal
edge of $W(s_1, t, r_1)$ crosses column $\ell$, then by subdividing
the edges of $H_{\ell}$ we can shift of $r-r_1$ steps the diagonal
edges of $W(s_1, t, r_1)$. If no diagonal edge of $W(s_1, t, r_1)$
crosses column $\ell$, then no diagonal edge is shifted. As an
example, consider the graph $W(5, 4, 3)$ in Figure
\ref{fig_horizontal_expansion}. If we expand horizontally the graph
from column $\ell=0$, then no diagonal edge crosses column $\ell$
and we obtain a good Eulerian subgraph $W(7, 4, 3)$. If we expand
horizontally the graph from column $\ell=2$, then every diagonal
edge crosses column $\ell$ and we obtain a good Eulerian subgraph
$W(7, 4, 5)$.\end{proof}

\begin{figure}[htbp]
\centering
\begin{minipage}{12cm}
\includegraphics[width=12cm]{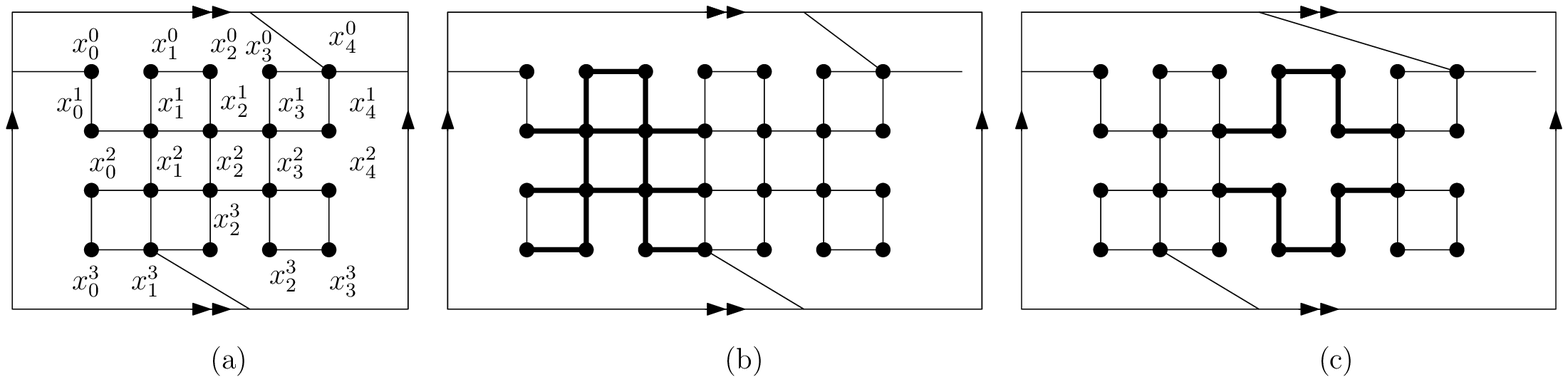}
\end{minipage}
\caption{A good Eulerian subgraph: (a) $W(5, 4, 3)$; $(b)$ $W(7, 4,
3)$; (c) $W(7, 4, 5)$. The graphs $W(7, 4, 3)$ and $W(7, 4, 5)$ are
obtained from $W(5, 4, 3)$ by an horizontal expansion from column
$0$ and column $2$, respectively.} \label{fig_horizontal_expansion}
\end{figure}

\subsection{Constructions of good Eulerian subgraphs.}\label{subsec_construction}

We apply the lemmas described in Section \ref{subsec_technical} to
construct a good Eulerian subgraph $W(s, t, r)$. It is
straightforward to see that the existence of loops in $X(s, t, r)$
excludes the existence of a good Eulerian subgraph $W(1, t, r)$ and
$W(s, 1, 0)$. Analogously, the existence of horizontal parallel
edges in $X(2, t, r)$ excludes the existence of a good Eulerian
subgraph $W(2, t, r)$ with $t$ odd and $W(2, t, 1)$ with $t$ even,
$t>2$, (see Case $2$ in the proof of Lemma
\ref{pro_construction_goodEuleriansubgraph} for a good Eulerian
subgraph $W(2, 2, 1)$ and $W(2, t, 0)$ with $t$ even). Hence we can
consider $s\geq 3$ and $(t, r)$$\neq (1, 0)$. The following hold.

\begin{proposition}\label{pro_good_t=1}
The graph  $X(s, 1, r)$, $r\neq 0$, possesses a good Eulerian
subgraph, unless $s=6m+5$, with $m\geq 0$, and $r\in\{2, s-2,
(s+1)/2, (s-1)/2\}$.
\end{proposition}

\begin{proof} By Proposition \ref{pro_r0}, the graph $X(s, 1, r)$
can be represented as the circulant multigraph $Cir(st; \pm 1, \pm
r)$. By Proposition \ref{pro_gcd_1}, the graph $X(s, 1, r)$
corresponds to the generalized Petersen graph $I(s, r, 1)$ or $G(s,
r)$. In particular, the graph $X(6m+5, 1, 2)$ corresponds to the
generalized Petersen graph $G(6m+5, 2)$. Hence  $X(s, 1, r)$ has a
good Eulerian subgraph, unless it is isomorphic to $X(6m+5, 1, 2)$,
since Theorem \ref{th_Alspach} and \ref{th_eulerian_subgraphs} hold.
By Proposition \ref{pro_iso_X}, the graphs that are isomorphic to
$X(6m+5, 1, 2)$ are $X(6m+5, 1, r')$, where $r'\in\{2, 6m+3\}$ or
$r'\equiv \pm 2^{-1}\pmod {6m+5}$, i.e., $r'\in\{3m+3, 3m+2\}$,
since $r'< 6m+5$.\end{proof}


We can construct a good Eulerian subgraph $W(s, 1, r)$, $r\neq 0$,
without using Theorem \ref{th_Alspach}. More specifically, by
Proposition \ref{pro_adam_iso2} the graph $X(s, 1, r)$, with $r\neq
0$, is isomorphic to the graph $X(s/\gcd(s, r), \gcd(s, r), r')$,
where $r'\equiv\pm r^{-1}\pmod s$. For $r\neq 0$ and $\gcd(s, r)>1$,
a construction of a good Eulerian subgraph can be found in the proof
of Lemma \ref{pro_construction_goodEuleriansubgraph}. We can also
provide an ad hoc construction for the case $\gcd(s, r)=1$, but we
prefer to omit this construction, since the existence of a good
Eulerian subgraph $W(s, 1, r)$, $r\neq 0$, is known (see Proposition
\ref{pro_good_t=1}) and the construction is based on the method of
Lemma \ref{pro_construction_goodEuleriansubgraph}. We will show that
the graph $X(6m+5, 1, 2)$, $m\geq 0$, has no good Eulerian subgraph,
i.e., the generalized Petersen graph is not Hamiltonian. The
following statement is a consequence of Proposition
\ref{pro_good_t=1} and it will be used in the proof of Lemma
\ref{pro_construction_goodEuleriansubgraph}.

\begin{proposition}\label{pro_good_gcd_sr}
The graph $X(s, t, r)$, with $s\geq 3$, $t>1$ and $\gcd(s, r)=1$ has
a good Eulerian subgraph.
\end{proposition}

\begin{proof} By Proposition \ref{pro_circulant}, the graph $X(s, t,
r)$ can be represented as the circulant graph $Cir(st; \pm t, \pm
k)$, where $\gcd(k, t)=1$ and $k\equiv r\pmod s$. By Proposition
\ref{pro_adam_iso2}, the graph $X(s, t, r)$ is isomorphic to the
graph $X(st, 1, r')$, with $r'\neq 0$, since $\gcd(s, r)=1$. If
$st\not\equiv 5\pmod 6$, then the assertion follows from Proposition
\ref{pro_good_t=1} (see Proposition \ref{pro_adam_iso2}). Consider
$st\equiv 5\pmod 6$. We show that $X(s, t, r)$ is not isomorphic to
$X(6m+5, 1, 2)$, $m\geq 0$. Suppose, on the contrary, that $X(s, t,
r)$ is isomorphic to $X(6m+5, 1, 2)$. Then $X(st, 1, r')$$=X(6m+5,
1, r')$, where $r'\in\{2$, $st-2$, $(st+1)/2$, $(st-1)/2\}$ (see
Proposition \ref{pro_good_t=1}). By Proposition \ref{pro_adam_iso2},
the integer $r'$ satisfies the relation $r'\equiv\pm
tk^{-1}\pmod{st}$. Whence $t$ is a divisor of $r'$. That yields a
contradiction, since $r'\in\{2$, $st-2$, $(st+1)/2$, $(st-1)/2\}$
and $t$ is coprime with the integers in $\{2$, $st-2$, $(st+1)/2$,
$(st-1)/2\}$.\end{proof}

\begin{lemma}\label{pro_construction_goodEuleriansubgraph}
Let $s\geq 3$, $t\geq 2$ and $0\leq r\leq s-1$. There exists a good
Eulerian subgraph $W(s, t, r)$, unless $s$ is odd and $(t, r)=(2,
0)$.
\end{lemma}

\begin{proof} We treat separately the cases: $t=3$; $s$, $t$ even;
$s$ even, $t$ odd, $t\geq 5$; $s$ odd, $t$ even; $s$, $t$ odd,
$t\geq 5$. When we will speak of ``vertical'' and ``horizontal''
expansion we refer implicitly to Lemma \ref{lemma_vertical} and
\ref{lemma_horizontal2}, respectively.

\smallskip

\noindent\textbf{Case 1: $t=3$.} This case is treated in Section
\ref{sec_appendix}, since it requires a quite long description.

\noindent\textbf{Case 2: $s$ even, $t$ even.} The graph $W''(6, 8,
r)$ in Figure \ref{fig_pro1}(d) can be expanded vertically from row
$1$ and horizontally from column $2$. It yields a good Eulerian
subgraph $W''(s, t, r)$ for every $s, t$ even $s\geq 6$, $t\geq 8$.
It remains to construct a good Eulerian subgraph $W''(s, t, r)$ for
$s\geq 6$, $t=2$, $4$, $6$ and $W''(4, t, r)$ for $t\geq 2$, $t$
even. The graph $W'(2, 2, r)$ in Figure \ref{fig_pro1}(a) can be
expanded horizontally from column $0$ or $1$. It yields a good
Eulerian subgraph $W'(s, 2, r)$ for every $s$ even, $s\geq 2$. We
expand horizontally the graph $W''(4, 4, r)$ in Figure
\ref{fig_pro1}(b) and obtain $W''(s, 4, r)$ for every $s$ even,
$s\geq 4$. We rotate $W''(s, 4, r)$ by $90$ degrees clockwise
(around a vertex) and obtain a good Eulerian subgraph $W''(4, t, r)$
for every $t$ even, $t\geq 4$. We expand horizontally the graph
$W''(6, 6, r)$ in Figure \ref{fig_pro1}(c) from column $3$ and
obtain $W''(s, 6, r)$ for every $s$ even, $s\geq 6$.

\begin{figure}[htbp]
\centering
\begin{minipage}{10cm}
\includegraphics[width=10cm]{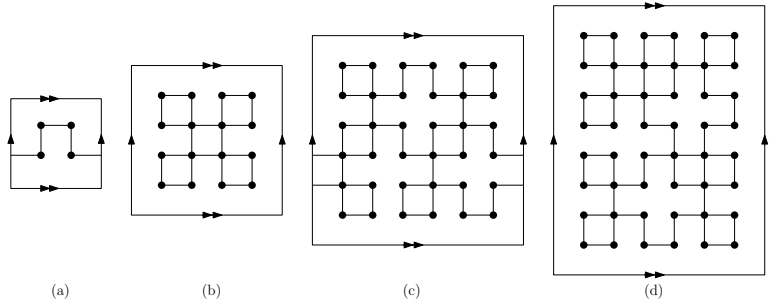}
\end{minipage}
\caption{A good Eulerian subgraph: (a) $W'(2, 2, r)$; (b) $W''(4, 4,
r)$; (c) $W'(6, 6, r)$; (d) $W''(6, 8, r)$.} \label{fig_pro1}
\end{figure}

\smallskip

\noindent\textbf{Case 3: $s$ even, $t$ odd, $t\geq 5$.}  The graph
$W'(6, 5, r)$ in Figure \ref{fig_vertical_expansion} can be expanded
vertically from row $2$ and horizontally from column $3$. It yields
a good Eulerian subgraph $W'(s, t, r)$ for every $s$ even, $s\geq
6$, $t$ odd, $t\geq 5$. It remains to construct $W(4, t, r)$ with
$t$ odd, $t\geq 5$, $0\leq r\leq 3$. Since $X(4, t, r)$ is
isomorphic to $X(4, t, 4-r)$, we can consider $0\leq r\leq 2$. A
good Eulerian subgraph for $W(4, t, 0)$, $t$ odd, $t\ge 5$, can be
obtained from $W(4, 3, 0)$ in Figure \ref{fig_430}(a) by a vertical
expansion from row $1$. The existence of a good Eulerian subgraph
$W(4, t, 1)$ follows from Proposition \ref{pro_good_gcd_sr}. By
Proposition \ref{pro_adam_iso2}, the graph $X(4, t, 2)$ is
isomorphic to the graph $X(2t, 2, r')$. By the results in Case $2$,
there exists a good Eulerian subgraph $W(2t, 2, r')$.

\smallskip

\noindent\textbf{Case 4: $s$ odd, $t$ even.} By Proposition
\ref{pro_adam_iso2}, the graph $X(s, t, r)$, with $r\neq 0$, is
isomorphic to the graph $X(st/\gcd(s, r), \gcd(s, r), r')$, with
$r'\neq 0$, or to $X(t, s, 0)$ if $r=0$. If $r\neq 0$ and $\gcd(s,
r)=1$ or $3$, then the existence of a good Eulerian subgraph follows
from Proposition \ref{pro_good_gcd_sr} or from the results in Case
$1$, respectively. Note that $st/\gcd(s, r)\geq 4$, since $t$ is
even and $0<r\neq s-1$. Hence, for $\gcd(s, r)\geq 5$, the existence
of a good Eulerian subgraph follows from Case $3$. Consider $r=0$.
There is no good Eulerian subgraph $W(s, 2, 0)$, because of the
existence of parallel vertical edges. Consider $t\geq 4$. As
remarked, the graph $X(s, t, 0)$ is isomorphic to the graph $X(t, s,
0)$. For $s\geq 5$ the existence of a good Eulerian subgraph $W(t,
s, 0)$ follows from the results in Case $3$. The existence of a good
Eulerian subgraph $W(t, 3, 0)$ follows from Case $1$.

\smallskip

\noindent\textbf{Case 5: $s$ odd, $t$ odd, $t\geq 5$.} A good
Eulerian subgraph $W(s, t, 0)$ can be obtained from the graph $W(3,
3, 0)$ in Figure \ref{fig_430}(a).
If $r\in\{1, 2\}$, then the  existence of a good Eulerian subgraph
follows from Proposition \ref{pro_good_gcd_sr}.  Consider $3\leq
r\leq s-3$ and $s\geq 7$. Since $X(s, t, r)$ is isomorphic to $X(s,
t, s-r)$ and $s$ is odd, we can construct a good Eulerian subgraph
$W(s, t, r)$ for every $s$, $r$ odd, $s\geq 7$, $3\leq r\leq s-4$.
The graph $W(7, 5, 3)$ in Figure \ref{fig_733}(c) can be expanded
horizontally from column $4$ and vertically from row $1$ (or $2$).
It yields a good Eulerian subgraph $W(s, t, 3)$ for every $s$, $t$
odd, $s\geq 7$, $t\geq 5$. Since $s-r+3\geq 7$, we can consider the
graph $W(s-r+3, t, 3)$ arising from $W(7, 5, 3)$ in Figure
\ref{fig_733}(c). We expand horizontally the graph $W(s-r+3, t, 3)$
from column $2$ and obtain a good Eulerian subgraph $W(s, t, r)$ for
every $s$, $t$, $r$ odd, $s\geq 7$, $t\geq 5$ and $3\leq r\leq
s-4$.\end{proof}

\begin{proposition}\label{pro_good_GP}
The graph $X(6m+5, 1, 2)$, $m\geq 0$, has no good Eulerian subgraph.
Consequently, the generalized Petersen graph $G(6m+5, 2)$ has no
Hamiltonian cycle.
\end{proposition}

\begin{proof} We give a sketch of the proof by showing that $X(5, 1,
2)$ has no good Eulerian subgraph. Suppose, on the contrary, that
$W$ is a good Eulerian subgraph of $X(6m+5, 1, 2)$. Since the unique
horizontal layer of $W$ has an odd number of vertices, the graph $W$
contains at least one path $P_{2j+1}$ consisting of $2j$ horizontal
edges. It is possible to prove that $2j=2$ (if $2j>2$, then $W$ is
not good). Without loss of generality we can set $P_{2j+1}=(x^0_0,
x^0_1, x^0_2)$. Whence $[x^0_3, x^0_4]\in E(W)$ and no other
horizontal edge of $X(5, 1, 2)$ belongs to $E(W)$. Moreover,
$[x^0_1, x^0_3]$,$[x^0_1, x^0_4]$  are edges of $W$, since $W$ is
admissible and $x^0_1$ is $4$-valent in $W$. Whence $[x^0_0,
x^0_2]\in E(W)$
and each admissible tour of $W$ contains the component $A=(x^0_3,
x^0_4$, $x^0_4$, $x^0_3)$. That yields a contradiction, since $A$ is
not a spanning subgraph of $X(6m+5, 1, 2)$. Hence $X(5, 1, 2)$ has
no good Eulerian subgraph. By Theorem \ref{th_eulerian_subgraphs},
the graph $G(5, 2)$ has no Hamiltonian cycle. The proof can be
generalized to the case $G(6m+5, 2)$ with $m>0$.\end{proof}

\section{Proof of the main Theorem; characterization of Hamiltonian
$I$-graphs.}\label{sec_characterization_ham_Igraph}

Now we are ready to prove the main Theorem.


\begin{proof}[Proof of Theorem~{\rm\ref{th_hamiltonian}}] By Theorem
\ref{th_Alspach}, a generalized Petersen graph is Hamiltonian if and
only if it is not isomorphic to $G(6m+5, 2)$, $m\geq 0$. We prove
that a proper $I$-graph is Hamiltonian. By Lemma
\ref{pro_Igraph_special}, a proper $I$-graph $I(n, p, q)$ is special
and its associated quartic graph $X$ is the circulant graph $Cir(n;
p, q)$. By Lemma \ref{pro_circulant2}, the graph $Cir(n; p, q)$ can
be represented as the graph $X(s, t, r)$, where $t=\gcd(n, q)$,
$s=n/t\geq 3$, $r\equiv\pm p(q/t)^{-1}\pmod s$ and $(t, r)$$\neq (2,
0)$ for odd values of $s$. By Lemma
\ref{pro_construction_goodEuleriansubgraph}, the graph $X(s, t, r)$
has a good Eulerian subgraph. The assertion follows from Theorem
\ref{th_eulerian_subgraphs}.\end{proof}


By Theorem \ref{th_eulerian_subgraphs} and Lemma
\ref{pro_construction_goodEuleriansubgraph}, we can extend the
result of Theorem \ref{th_hamiltonian}, about the existence of a
Hamiltonian cycle, to the special generalized $I$-graphs.

As a consequence of Theorem \ref{th_hamiltonian}, a proper $I$-graph
is $3$-edge-colorable or, equivalently, $1$-factorizable (because it
is cubic and Hamiltonian). A widely studied property for
$1$-factorizable graphs is the property of admitting a perfect
$1$-factorization. We recall that a $1$-factorization is perfect if
the union of any pair of distinct $1$-factors is a Hamiltonian
cycle. This property has been investigated for the complete graph:
the complete graph $K_{n}$ has a perfect $1$-factorization when
$n/2$ or $n-1$ is a prime (see \cite{An1} \cite{An2} and \cite{Ko}).
Partial results are also known for generalized Petersen graphs:
$G(n,k)$ admits a perfect $1$-factorization when $(n,k)=(3, 1)$;
$(n,k)=(n,2)$ with $n\equiv 3, 4\pmod 6$; $(n,k)=(9,3)$; $(n,k)=(3d,
d)$ with $d$ odd; $(n,k)=(3d, k)$ with $k>1$, $d$ odd, $3d$ and $k$
coprime (see \cite{BM}). So, it is quite natural to extend the same
problem to proper $I$-graphs.

Some further problems can be considered: the generalization of the
existence of good Eulerian tour to other graph bundles of a cycle
over a cycle, the characterization of Hamiltonian $GI$-graphs or of
Hamilton-laceable $I$-graphs. In \cite{DW}, the authors proved by a
computer search that all bipartite connected $I$-graphs on $2n\leq
200$ vertices are Hamilton-laceable.

\vspace{\baselineskip} \noindent {\bf Acknowledgements:} The authors
would like to thank Arjana \v{Z}itnik for careful reading of various
versions of this paper and for many useful suggestions. Research
supported in part by the ARRS Grants P1-0294, N1-0032, and J1-6720.

\section{Appendix. Proof of Lemma ~{\rm\ref{pro_construction_goodEuleriansubgraph}}}\label{sec_appendix}

\begin{proof}[Case $1$, $t=3$] We expand horizontally the graph $W(3, 3, 0)$ in
Figure \ref{fig_430}(a) from column $0$ and obtain a good Eulerian
subgraph $W(s, 3, 0)$ for every $s$ odd, $s\geq 3$. A good Eulerian
subgraph $W(s, 3, 0)$ with $s$ even can be obtained from the graphs
$W(4, 3, 0)$ and $W(6, 3, 0)$ in  Figure \ref{fig_430}(b)-(c). As an
example, the graph $W(8, 3, 0)$ in Figure \ref{fig_830}(a) has been
obtained by connecting two copies of the graph $W(4, 3, 0)$. The
graph $W(10, 3, 0)$ in Figure \ref{fig_830}(b) has been obtained by
connecting the graphs $W(4, 3, 0)$ and $W(6, 3, 0)$. For $r=1$ the
existence of a good Eulerian subgraph $W(s, 3, 1)$ follows from
Proposition \ref{pro_good_gcd_sr}. Hence we can consider $2\leq
r\leq s/2$, since $X(s, 3, r)$ is isomorphic to $X(s, 3, s-r)$. The
graph $W(4, 3, 2)$ in Figure \ref{fig_830}(c) can be expanded
horizontally from column $3$. It yields a good Eulerian subgraph
$W(s, 3, 2)$ for every $s$ even, $s\geq 4$. Since $s-r+2\geq 4$, we
can consider the graph $W(s-r+2, 3, 2)$ obtained from $W(4, 3, 2)$
in in Figure \ref{fig_830}(c). We expand horizontally $W(s-r+2, 3,
r)$ from column $1$ and obtain a good Eulerian subgraph $W(s, 3, r)$
for every $s$, $r$ even, $s\geq 4$, $2\leq r\leq s/2$. Analogously,
the graphs $W(6, 3, 3)$, $W(8, 3, 3)$ and $W(10, 3, 5)$ in Figure
\ref{fig_633} yield a good Eulerian subgraph $W(s, 3, r)$ for every
$s$ even, $r$ odd, $3\leq r\leq s/2$. More specifically, we expand
horizontally the graph $W(8, 3, 3)$ from column $7$ and obtain a
good Eulerian subgraph $W(s, 3, 3)$ for every even integer $s\geq
8$. The graph $W(10, 3, 5)$ can be expanded horizontally from column
$9$ (or $0$). It yields a good Eulerian subgraph $W(s, 3, 5)$ for
every even integer $s$, $s\geq 10$. Since $s-r+5\geq 10$, we can
consider the graph $W(s-r+5, 3, 5)$ obtained from $W(10, 3, 5)$ in
Figure \ref{fig_633}(c). We expand $W(s-r+5, 3, 5)$ from column $4$
and obtain a good Eulerian subgraph $W(s, 3, r)$ for every $s$ even,
$s\geq 10$, $r$ odd, $5\leq r\leq s/2$.

Consider $s$ odd, $s\geq 5$. The graph $W(5, 3, 2)$ Figure
\ref{fig_532}(a) can be expanded horizontally from column $4$. It
yields a good Eulerian subgraph $W(s, 3, 2)$ for every $s$ odd,
$s\geq 5$. Analogously, the graph $W(9, 3, 4)$ in Figure
\ref{fig_532}(b) yields a good Eulerian subgraph $W(s, 3, 4)$ for
every $s$ odd, $s\geq 9$. The graph $W(13, 3, 6)$ in Figure
\ref{fig_532}(c) can be expanded horizontally from column $2$ and
column $10$. It yields a good Eulerian subgraph $W(2r+1, 3, r)$ with
$r$ even, $6\leq r\geq s/2$. Since $s-2r+1\geq 0$, we can expand
$W(2r+1, 3, r)$ from column $2r$ and find a good Eulerian subgraph
$W(s, 3, r)$ for every $s$ odd, $s\geq 13$, $r$ even, $r\geq 6$. It
remains to construct a good Eulerian subgraph $W(s, 3, r)$ with $s$,
$r$ odd, $s\geq 7$, $3\leq r\leq s/2$. We use he graph $W(7, 3, 3)$
in Figure \ref{fig_733}(a) to construct a good Eulerian subgraph
$W(2r+1, 3, r)$ with $r$ odd, $r\geq 3$. As an example, the graph
$W(11, 3, 5)$ in Figure \ref{fig_733}(b) has been obtained by
expanding horizontally the graph $W(7, 3, 3)$ from column $r=3$ and
$s-1=6$ and by adding new diagonal edges. If we iterate the process,
then we obtain a good Eulerian subgraph $W(2r+1, 3, r)$ with $r$
odd, $r\geq 3$. The graph $W(2r+1, 3, r)$ thus obtained can be
expanded horizontally from column $2r$. It yields a good Eulerian
subgraph $W(s, 3, r)$ for every $s$, $r$ odd, $s\geq 7$, $3\leq
r\leq s/2$.\end{proof}

\begin{figure}[htbp]
\centering
\begin{minipage}{7cm}
\includegraphics[width=7cm]{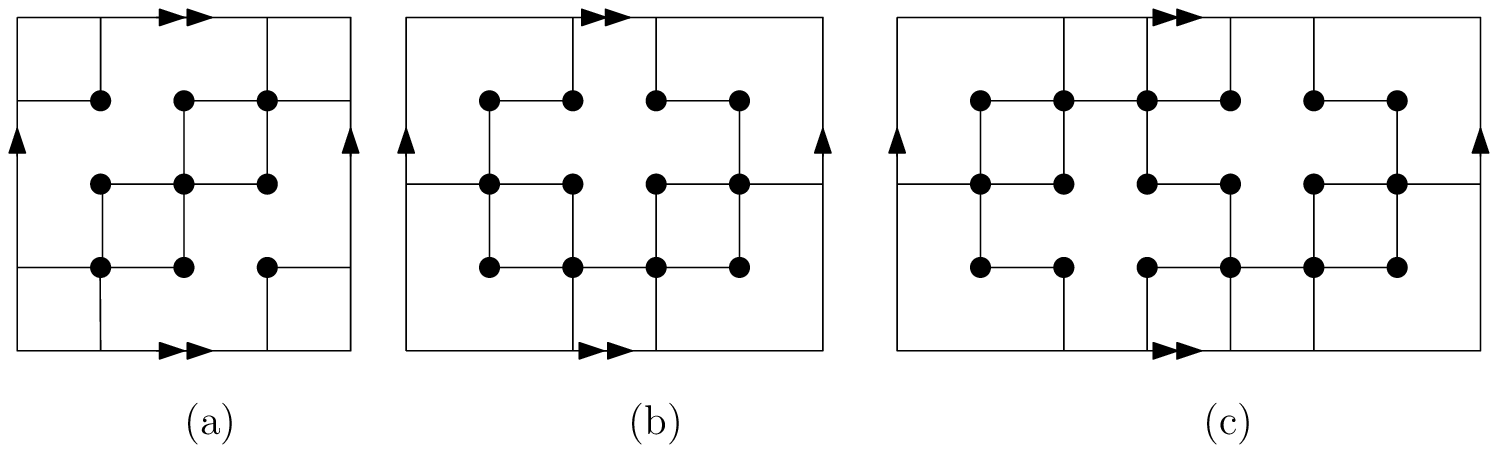}
\end{minipage}
\caption{A good Eulerian subgraph: (a) $W(3, 3, 0)$; (b) $W(4, 3,
0)$; (c) $W(6, 3, 0)$.} \label{fig_430}
\end{figure}

\begin{figure}[htbp]
\centering
\begin{minipage}{11cm}
\includegraphics[width=11cm]{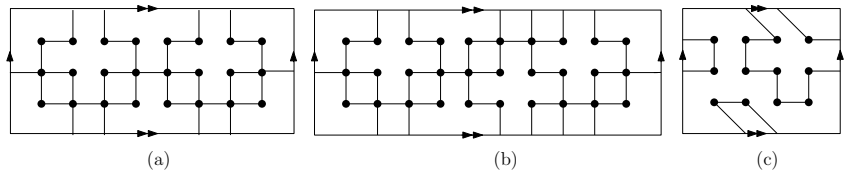}
\end{minipage}
\caption{A good Eulerian subgraph: (a) $W(8, 3, 0)$; (b) $W(10, 3,
0)$; (c) $W(4, 3, 2)$.} \label{fig_830}
\end{figure}

\begin{figure}[htbp]
\centering
\begin{minipage}{12cm}
\includegraphics[width=12cm]{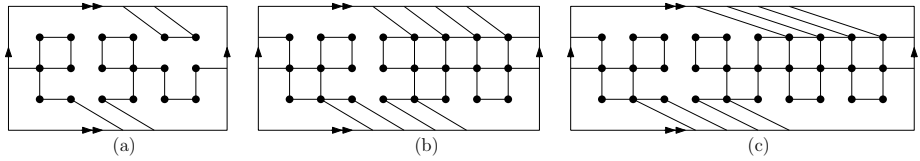}
\end{minipage}
\caption{A good Eulerian subgraph: (a) $W(6, 3, 3)$; (b) $W(8, 3,
3)$; (c) $W(10, 3, 5)$.} \label{fig_633}
\end{figure}

\begin{figure}[htbp]
\centering
\begin{minipage}{8cm}
\includegraphics[width=8cm]{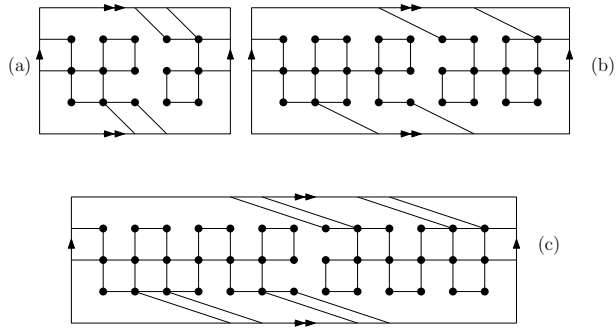}
\end{minipage}
\caption{A good Eulerian subgraph: (a) $W(5, 3, 2)$; (b) $W(9, 3,
4)$; (c) $W(13, 3, 6)$} \label{fig_532}
\end{figure}

\begin{figure}[htbp]
\centering
\begin{minipage}{12cm}
\includegraphics[width=12cm]{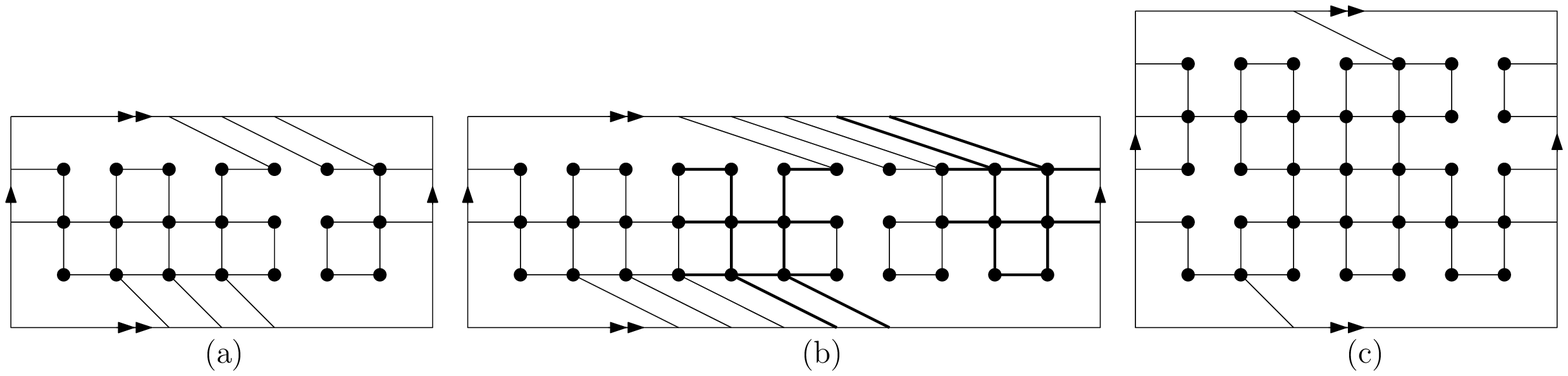}
\end{minipage}
\caption{A good Eulerian subgraph: (a) $W(7, 3, 3)$; (b) $W(11, 3,
5)$; (c) $W(7, 5, 3)$. To obtain the graph $W(11, 3, 5)$ we expanded
horizontally the graph $W(7, 3, 3)$ from column $r=3$ and column
$s-1=6$, then we added new diagonal edges (see the bold edges).}
\label{fig_733}
\end{figure}

\end{document}